\documentclass[11pt]{amsart}
\usepackage{amssymb,amsmath}
\newcommand{\dist}{{\rm dist}}

\def\Bbb{\mathbb}
\def\Cal{\mathcal}
\def\Dt{\partial_t}

\def\eb{\varepsilon}

\def\R {\mathbb{R}}

\def\<{\left<}
\def\>{\right>}

\def\Nx{\nabla_x}
\def\Dx{\Delta_x}
\def\({\left(}
\def\){\right)}
\def\supp{\operatorname{supp}}
\def\sgn{\operatorname{sgn}}

\newtheorem{proposition}{Proposition}[section]
\newtheorem{theorem}[proposition]{Theorem}
\newtheorem{corollary}[proposition]{Corollary}
\newtheorem{lemma}[proposition]{Lemma}

\theoremstyle{definition}
\newtheorem{definition}[proposition]{Definition}

\newtheorem{remark}[proposition]{Remark}

\numberwithin{equation}{section}

\def \au {\rm}
\def \ti {\it}
\def \jou {\rm}
\def \bk {\it}
\def \no#1#2#3 {{\bf #1} (#3), #2.}
\def \eds#1#2#3 {#1, #2, #3.}
\begin{document}
\title[] {Global attractor and stabilization for a coupled PDE-ODE system}

\author[M.Efendiev and S.Zelik]
{Messoud Efendiev and Sergey Zelik}


\address{  GSF/Technical University Munich\newline Center of
Math.Sciences Boltzmann str.3\newline 85747 Garching/Munich Germany
\newline\phantom{eqqog}\newline
  University of Surrey\newline Department of Mathematics\newline
Guildford, GU2 7XH, United Kingdom}
\email{efendiev@math.fu-berlin.de, s.zelik@surrey.ac.uk}

\begin{abstract} We study the asymptotic behavior of solutions of
one coupled PDE-ODE system arising in mathematical biology as a
model for the development of a forest ecosystem.
\par
 In the case
where the ODE-component of the system is monotone, we establish
the existence of a smooth global attractor of finite Hausdorff and
fractal dimension.
\par
The case of the non-monotone ODE-component is much more
complicated. In particular, the set of equilibria becomes
non-compact here and contains a huge number of essentially
discontinuous solutions. Nevertheless, we prove the stabilization
of any trajectory to a single equilibrium if the coupling constant
is small enough.
\end{abstract}
\maketitle
\tableofcontents

\section{Introduction}\label{s0}

We study the following coupled ODE-PDE system
\begin{equation}\label{0.main}
\begin{cases}
\Dt^2 v+\varphi(v)\Dt v+f(v)=\alpha w\\
\Dt w-\Dx w+w=v,\ \ \partial_nw\big|_{\partial\Omega}=0
\end{cases}
\end{equation}
in a bounded smooth domain $\Omega\subset\R^n$, $n\le3$. Here
$(v,w)=(v(t,x),w(t,x))$ are unknown functions, $\Dx$ is a
Laplacian with respect to $x$, $\alpha>0$ is a given parameter and
$f$ and $\varphi$ are given nonlinearities which are assumed to
satisfy some natural assumptions formulated in Section \ref{s1}.
\par
Our interest to that problem is motivated by the following system
arising in the mathematical biology:
\begin{equation}\label{0.forest}
\begin{cases}
\Dt u=\beta\delta w-\gamma(v)u -f u,\\
\Dt v=fu-hv,\\
\Dt w-d\Dx w+\beta w=\alpha v,\ \ \partial_nw\big|_{\partial\Omega}=0,
\end{cases}
\end{equation}
where $\alpha,\beta,\delta,d,f,h$ are given positive parameters
and $\gamma(v)$ is a given nonlinearity.  This system has been
introduced in \cite{K} in order to describe the development of a
forest ecosystem (the unknown functions $u$, $v$ and $w$ are the densities of yang trees, old trees and the seeds  respectively and the given nonlinearity $\gamma(v)$ describes the mortality of yang trees in dependence of the density of the old ones) and has been recently studied analytically and numerically in \cite{Y1,Y2,Y3}.
Expressing $u=(\Dt v+hv)/f$ from the second equation of \eqref{0.forest} and insert it
into the first one, we end up with the system of the type
\eqref{0.main} with respect to the variables $(v,w)$.
\par
The main aim of the present paper is to study the long-time
behavior of solutions of \eqref{0.main} using the ideas and methods of
the attractors theory. From the mathematical point of view, the
problem considered is a coupled system of a second order ODE with
a linear PDE (heat-like equation). Heuristically, it is clear that
the dynamics of such coupled dissipative systems should depend drastically on
the monotonicity properties of the ODE component. In the case
where this ODE is "monotone", i.e., it cannot produce the internal
instability (and all of the instability is driven by the coupling with the PDE
component), one expects the asymptotic compactness and the
existence of a smooth finite-dimensional global attractor with
"good" properties. In contrast to that, in the non-monotone case,
the ODE-instability may produce the asymptotic discontinuities and
even may completely destroy the initially smooth spatial profile.
Thus, in that case, the smoothing effect from the PDE component is
not strong enough in order to suppress the development of
discontinuities provided by the internal instabilities of the
ODE-component and, as a result, spatial discontinuities and extremely complicated (in a
sense, pathological) spatial structures may appear (see, for instance, \cite{AB} and \cite{pego} for the analysis of similar effects in the 1D finite visco-elasticity). We also mention that,
although the existence and uniqueness of a solution of
\eqref{0.forest} has been rigorously proved in the above mentioned
papers \cite{Y1,Y2}, very few has been done concerning the
asymptotic behavior of solutions as $t\to\infty$. To be more
precise, different types of $\omega$-limit sets of a single trajectory were considered
there (associated with the different choice of the topology in the
phase space) and their simplest properties were formulated, but
the even the question whether or not they are empty remained open.
As we will see below, some of them are indeed empty for the
most part of the trajectories if the monotonicity assumption is
essentially violated, see Remark \ref{Rem3.discont}

\par
In the present paper, we justify the above heuristics in a
mathematically rigorous way on the example of the ODE-PDE coupled
system \eqref{0.main}. In particular, we show (in Section \ref{s2}) that the
monotonicity arguments work  perfectly if
\begin{equation}\label{0.mon}
f'(v)\ge\kappa_0>0,\ \ v\in\R.
\end{equation}
In this case, problem \eqref{0.main} possesses indeed a smooth
global attractor $\Cal A$ of finite fractal dimension in the proper phase
space $\Phi_\infty$. Moreover, due to the presence of a global
Lyapunov function, this attractor generically has very nice
properties
(it is the so-called regular attractor in the terminology of Babin
and Vishik \cite{bv}). Namely, it is a finite collection of the
finite-dimensional unstable manifolds associated with the
equilibria:
$$
\Cal A=\cup_{u_0\in\Cal R}\Cal M^+_{u_0}
$$
(where $\Cal R$ denotes a (generically finite) set of equilibria of
problem \eqref{0.main} and $\Cal M_{u_0}^+$ is an unstable
manifold associated with the equilibrium $u_0\in\Cal R$, see
Section \ref{s2}). Moreover, every trajectory of \eqref{0.main}
converges {\it exponentially} to one of that equilibria.
We also mention that the first equation of \eqref{0.main} is a
{\it second} order ODE and, therefore, the monotonicity of $f$
does not automatically imply the absence of the internal
instability. For instance, the ODE
$$
y''+\varphi(y)y'+f(y)=h(t)
$$
may produce the non-trivial dynamics even if $f$ is monotone and
$\varphi$ is strictly positive, say for the case of a given
time-periodic external force $h$. By this reason, our proof of the
monotonicity of the ODE component is based on rather delicate
arguments related with the existence of the global Lyapunov
function and associated dissipative integrals, see Section
\ref{s2}.
\par
The case where the monotonicity assumption is violated (which is
considered in Section \ref{s3} occurs to be (as predicted by the
heuristics) much more complicated. In contrast to the monotone
case, there is a very few hope to develop a reasonable global
attractor theory here (no matter in a strong or weak topology of
the phase space), since, as a rule, even the equilibria set $\Cal R$
is already not compact in the strong topology of the phase space
and not closed in the weak topology. In addition, we see indeed a
huge (uncountable) number of well-separated essentially discontinuous equilibria
here, see Section \ref{s3}.
\par
Nevertheless, in the particular case of small coupling constant
$\alpha$, we succeed to give a complete description of the
equilibria set $\Cal R=\Cal R_\alpha$ and verify that every
trajectory of \eqref{0.main} converges as $t\to\infty$ to one of
that equilibria. We mention that the standard Lojasiewicz technique
for proving the stabilization seems to be non-applicable here even in the case of analytic
 non-linearities,
since the equilibria set is not compact in any reasonable
topology and the alternative technique of \cite{AB} and \cite{pego} (see also
\cite{KZ} where the pointwise stabilization for the non-smooth
temperature driven phase separation model is proved) also does not work here since it
is essentially based on the fact that the corresponding non-monotone ODE is a {\it first}
order scalar ODE and cannot be generalized to the case of higher order equations.
\par
 By this reason, we develop a new method of proving the
stabilization, based on the theory of non-autonomous perturbations
of regular attractors, see Appendix.
\par
Mention also that, as pointed out in \cite{Y1}, the solutions with
discontinuous densities are rather expected in a view of the forest ecosystem and a curve in $\Omega$ where the density has discontinuities is called ecotone boundary. However, as our result shows, this "curve" is typically not smooth (and even not continuous) and may have an extremely complicated structure.
\par
The paper is organized as follows. Section \ref{s1} is devoted to
the study of the analytical properties of problem \eqref{0.main}
such as existence and uniqueness, dissipative estimates in
different norm, etc. The case of monotone nonlinearity $f$ is
considered in Section \ref{s2}, in particular, the existence of
smooth regular attractor is proved here. In Section \ref{s3} we
deal with the non-monotone case and, in particular, prove here the
above mentioned stabilization result for the weakly coupled case.
Finally, the Appendix is devoted to the derivation of the key
estimate for our stabilization method which is, in turns, based on
the perturbation theory of regular attractors.
\par
To conclude, we mention that it would be interesting to consider
the regularization of \eqref{0.main} in a spirit of a damped wave
equation with displacement depending damping (see
\cite{p,pz1,pz2}):
\begin{equation}\label{0.sm}
\begin{cases}
\Dt^2 v+\varphi(v)\Dt v+f(v)-\eb\Dx v=\alpha w\\
\Dt w-\Dx w+w=v
\end{cases}
\end{equation}
with $0<\eb\ll1$. We return to that problem somewhere else.

\section{A priori estimates, existence and uniqueness}\label{s1}
We consider the following coupled system of a second order ODE
with a heat equation:
\begin{equation}\label{1.eqmain}
\begin{cases}
\Dt^2 v+\varphi(v)\Dt v+f(v)=\alpha w, \ v(0)=v_0,\ \Dt v(0)=v_0',\\
\Dt w-\Dx w +w=v,\
\partial_n w\big|_{\partial\Omega}=0, \ \ w\big|_{t=0}=w_0
\end{cases}
\end{equation}
in a bounded 3D domain $\Omega\subset\R^3$ with a smooth boundary.
Here, $(v,w)=(v(t,x),w(t,x))$ are unknown functions, $\Dx$ is a
Laplacian with respect to the variable $x$, $\alpha>0$ is a given
constant and $\varphi$ and $f$ are given nonlinearities, which
satisfy the following assumptions:
\begin{equation}\label{1.assum}
\begin{cases}
1.\ \ \varphi,f\in C^2(\R),\\
2.\ \ \varphi(v)\ge \beta_0>0,\\
3.\ \ f(v)v\ge-C+\gamma_0|v|^{2+\delta},\\
4.\ \ f'(v)\ge-K
\end{cases}
\end{equation}
for some positive constants $C$, $K$, $\beta_0$, $\delta$ and $\gamma_0$.
\par
Finally, we assume that the initial data $(v_0,v_0',w_0)$ is taken
from the $L^\infty(\Omega)$:
\begin{equation}\label{1.id}
(v_0,v_0',w_0)\in \Phi_\infty:=[L^\infty(\Omega)]^2\times[L^\infty(\Omega)\cap H^1(\Omega)].
\end{equation}
The aim of that section is to establish a number of basic a priori
estimates for that system which will allow us to verify the
existence and uniqueness of a solution and to study its behavior
as $t\to\infty$. We start with the following lemma which gives the
global Lyapunov function for that problem.

\begin{proposition}\label{Prop1.Lyap} Let the above assumptions
hold and let $(v(t),w(t))$ be a sufficiently regular solution of
problem \eqref{1.eqmain}.  Introduce a functional
\begin{equation}\label{1.lyap}
\Cal L(v,w):=\|\Dt v\|^2_{L^2}+2(F(v),1)-2\alpha(v,w)+
\alpha\|\Nx w\|^2+\alpha\|w\|^2_{L^2},
\end{equation}
where $F(v):=\int_0^vf(s)\,ds$ and $(\cdot,\cdot)$ is used for the
inner product in $L^2$. Then, the following equality holds
\begin{equation}\label{1.eq}
\frac d{dt}\Cal L(v(t),w(t))=-2(\varphi(v(t))\Dt v(t),\Dt v(t))-\alpha\|\Dt
w(t)\|^2_{L^2}.
\end{equation}
\end{proposition}
Indeed, multiplying the first and the second  equations of \eqref{1.eqmain} by $\Dt
v$  and $\alpha\Dt w$ respectively, taking a sum and
integrating over $\Omega$, we arrive at \eqref{1.eq}.

\begin{corollary}\label{Cor1.l2} Let the above assumptions hold
and let $(v(t),w(t))$ be a solution of \eqref{1.eqmain}. Then, the
following estimate holds:
\begin{equation}\label{1.l2bound}
\|\Dt v(t)\|_{L^2}+\|v(t)\|_{L^2}+\|w(t)\|_{H^1}\le
Q(\|(v,\Dt v,w)\|_{\Phi_\infty})
\end{equation}
for some monotone function $Q$ independent of $t$ and the solution
\end{corollary}
Indeed, according due to our dissipativity assumption
\eqref{1.assum}.3,
\begin{equation}\label{1.Fest}
F(u)\ge-C_1+\gamma_1|u|^{2+\delta}
\end{equation}
for some new constants $C_1$ and $\gamma_1$. Using this
inequality, we easily check that
\begin{equation}\label{1.Lbound}
\gamma_2(\|\Dt v\|^2_{L^2}+\|v\|^2_{L^2}+\|w\|^2_{H^1})-C_2\le\Cal L(v,w)\le Q(\|(v,\Dt v,w)\|_{\Phi_\infty})
\end{equation}
for some constants $\gamma_2,C_2>0$ and some monotone function
$Q$. Integrating now equation \eqref{1.eq} by $t$ and using that
$\varphi(v)\ge0$ and $\alpha>0$, we arrive at  \eqref{1.l2bound}.
\par
The next corollary gives the $L^2$-dissipation integral for that
problem.
\begin{corollary}\label{Cor1.dis} Let the above assumptions hold.
Then,
\begin{equation}\label{1.disint}
\int_0^\infty\|\Dt v(t)\|^2_{L^2}+\|\Dt w(t)\|^2_{L^2}\,dt\le
Q(\|(v_0,v_0',w_0)\|_{\Phi_\infty})
\end{equation}
for some monotone function $Q$.
\end{corollary}
Indeed, this estimate is an immediate corollary of \eqref{1.eq},
\eqref{1.Lbound} and the assumption that
$\varphi(v)\ge\beta_0>0$.
\par
We are now going to verify that the solution is globally bounded
in $\Phi_\infty$.
\begin{proposition}\label{Prop1.Linfbound} Let the above
assumptions hold. Then, the following estimate is valid:
\begin{equation}\label{1.Linf}
\|v(t)\|_{L^\infty}+\|\Dt v(t)\|_{L^\infty}+\|w(t)\|_{L^\infty\cap
H^1}\le Q(\|(v_0,v_0',w_0)\|_{\Phi_\infty})
\end{equation}
for some monotone function $Q$ independent of $t$ and the
solution.
\end{proposition}
\begin{proof} We first establish the $L^\infty$-bound for the
$w$-component. Indeed, according to Corollary \ref{Cor1.l2}, the right-hand side $v$ of the second equation
of \eqref{1.eqmain} is  bounded in $L^\infty(\R_+,L^2(\Omega))$.
Consequently, the standard regularity result for the heat equation
gives
\begin{equation}\label{1.winf}
\|w(t)\|_{L^\infty}\le
C\|w(0)\|_{L^\infty}e^{-t}+C\|v\|_{L^\infty(\R_+,L^2)}\le
Q(\|(v_0,v_0',w_0)\|_{\Phi_\infty}).
\end{equation}
(here we have implicitly used the restriction on the space
dimension).
\par
Thus, we only need to establish the $L^\infty$-bounds for the
$v$-component. To this end, we will use the $L^\infty$-bounds for
the $w$-component obtained before and will consider the equation
for the $v$-component as an ODE for every (almost every, being a pedant)
 {\it fixed} $x\in\Omega$. Indeed, let $y(t):=v(t,x_0)$. Then,
 this function solves
 \begin{equation}\label{1.ODE}
 y''(t)+\varphi(y)y'+f(y)=h(t)=h_{w,x_0}(t):=\alpha w(t,x_0).
 \end{equation}
Multiplying this equation by $y'+\eb y$, we have
\begin{equation}\label{1.odeest}
[(y')^2+2F(y)+2\eb yy'+2\eb R(y)]'+2(\varphi(y)-\eb)(y')^2+2\eb
f(y)y=2h(y'+\eb y),
\end{equation}
where $R(y):=\int_0^y\varphi(s)s\,ds$. Using that $\varphi(v)$ is
strictly positive and $f$ is dissipative, we deduce from this equation
that, for sufficiently small $\eb>0$
\begin{equation}\label{1.difin}
\frac d{dt}S(y,y')+\gamma((y')^2+y^2)\le C(|h(t)|^2+1),
\end{equation}
where
$$
S(y,y'):=(y')^2+2F(y)+2\eb yy'+2\eb R(y).
$$
and $\gamma$ is positive.
Moreover,
$$
\eb_0(y^2+(y')^2)-C\le S(y,y')\le Q(y^2+(y')^2)
$$
for some positive $\eb_0$ and $C$ and monotone $Q$. Applying
the Gronwall lemma to inequality \eqref{1.difin}, we conclude that
\begin{equation}\label{1.disinfest}
y(t)^2+(y'(t))^2\le Q(y(0)^2+y'(0)^2+\|h\|_{L^\infty(\R_+)}^2)
\end{equation}
for some monotone function $Q$ which is independent of $t$ and
$y$, see eg, \cite{pata}. Taking now the supremum with respect to
all $x_0\in\Omega$ and using \eqref{1.winf} for estimating $h$, we
deduce estimate \eqref{1.Linf} and finish the proof of the
proposition.
\end{proof}

We are now ready to verify the existence and uniqueness of a
solution for the problem \eqref{1.eqmain}.

\begin{definition}\label{Def1.sol} A pair of functions $(v(t),w(t))$ is a solution
of problem \eqref{1.eqmain} if
$$
(v(t),\Dt v(t),w(t))\in\Phi_{\infty}
$$
 for every $t\ge0$ and \eqref{1.eqmain}
is satisfied in the sense of distributions.
\end{definition}
Note that, from the first equation of \eqref{1.eqmain}, we see
that $\Dt^2 v(t)\in L^\infty(\Omega)$. Therefore $v(t)\in
W^{2,\infty}([0,T],L^\infty(\Omega))$ and the initial data for $v$
is well-defined. Analogously, the $w$-component is continuous as a
function with values, say, in $L^2(\Omega)$ and the initial data
is again well-defined.

\begin{theorem} \label{Th1.main} Let the above assumptions hold. Then,
for every $(v_0,v_0',w_0)\in\Phi_{\infty}$,  problem
\eqref{1.eqmain} possesses a unique solution in the sense of
Definition \ref{Def1.sol} and this solution satisfies estimate
\eqref{1.Linf}. Moreover, any two solutions $(v_1(t),w_1(t))$ and
$(v_2(t),w_2(t))$ satisfy the following estimate:
\begin{multline}\label{1.difest}
\|(v_1(t),\Dt v_1(t),w_1(t))-(v_2(t),\Dt
v_2(t),w_2(t))\|_{\Phi_{\infty}}\le\\\le Ce^{Kt}\|(v_1(0),\Dt v_1(0),w_1(0))-(v_2(0),\Dt
v_2(0),w_2(0))\|_{\Phi_{\infty}},
\end{multline}
where positive constants $C$ and $K$ depend only on the norms of
the initial data.
\end{theorem}
\begin{proof} Let us first verify the uniqueness and estimate
\eqref{1.difest}. Indeed, let
$$
(v(t),w(t)):=(v_1(t),w_1(t))-(v_2(t),w_2(t)).
$$
 Then, these
functions solve
\begin{equation}\label{1.eqdif}
\begin{cases}
\Dt^2v+\varphi(v_1)\Dt v+[\varphi(v_1)-\varphi(v_2)]\Dt
v_2+[f(v_1)-f(v_2)]=\alpha w,\\
\Dt w-\Dx w+w=v.
\end{cases}
\end{equation}
Multiplying now the first equation of  \eqref{1.eqdif} by $\Dt
v+\eb v$, $\eb>0$ is a small positive number, using the fact that $v_i,\Dt
v_i$ are globally bounded in $L^\infty$ and applying the Gronwall inequality in a standard
way (without integration by $x$!), we conclude that
\begin{multline}\label{1.difinf}
\|\Dt v(t)\|_{L^\infty(\Omega)}^2+\|v(t)\|_{L^\infty}^2\le\\\le
Ce^{Kt}(\|\Dt
v(0)\|_{L^\infty(\Omega)}^2+\|v(0)\|_{L^\infty(\Omega)}^2)+C\int_0^te^{K(t-s)}\|w(s)\|_{L^\infty(\Omega)}^2\,ds
\end{multline}
for some positive constants $C$ and $K$ depending only on the
$L^\infty$-norms of $v_i$ and $\Dt v_i$. Furthermore, due to the
maximum principle for the heat equation, we have the estimate
\begin{equation}\label{1.winf1}
\|w(t)\|_{L^\infty}\le
e^{-t}\|w(0)\|_{L^\infty}+\int_0^te^{-(t-s)}\|v(s)\|_{L^\infty}\,ds.
\end{equation}
Inserting this estimate into the right-hand side of
\eqref{1.difinf}, we arrive at
\begin{multline}\label{1.difinf1}
\|\Dt v(t)\|_{L^\infty(\Omega)}^2+\|v(t)\|_{L^\infty}^2\le\\\le
C'e^{Kt}(\|\Dt
v(0)\|_{L^\infty(\Omega)}^2+\|v(0)\|_{L^\infty(\Omega)}^2+
\|w(0)\|^2_{L^\infty})+C'\int_0^te^{K(t-s)}\|v(s)\|_{L^\infty(\Omega)}^2\,ds.
\end{multline}
Applying again the Gronwall inequality to that relation, we
conclude that
\begin{equation}\label{1.difinf2}
\|\Dt v(t)\|_{L^\infty(\Omega)}^2+\|v(t)\|_{L^\infty}^2\le
C'e^{2Kt}(\|\Dt
v(0)\|_{L^\infty(\Omega)}^2+\|v(0)\|_{L^\infty(\Omega)}^2+\|w(0)\|^2_{L^\infty}).
\end{equation}
This estimate, together with \eqref{1.winf1}, give the desired
$L^\infty$-estimate for the triple $(v,\Dt v,w)$. In order to
finish the proof of estimate \eqref{1.difest}, it remains to note
that the  desired estimate $H^1$-norm of the $w$-component is
immediate, since the $L^\infty$-control for the right-hand side of
the heat equation for $w$ is already obtained. Thus, the
uniqueness and Lipschitz continuity \eqref{1.difest} are proved.
\par
So, we only need to prove the existence of a solution. It can be
done in a standard way, based on a priori estimate \eqref{1.Linf},
using  the Banach fixed point theorem for proving the existence
of a local solution  and estimate \eqref{1.Linf} for extending
this solution globally in time, see eg, \cite{henry} for the details.
\end{proof}
Our next aim is to establish the basic dissipative estimate in the
phase space $\Phi_{\infty}$.

\begin{theorem}\label{Th1.disinf} Let the above assumptions
hold. Then, a solution $(v(t),w(t))$ of problem \eqref{1.eqmain}
satisfies the following dissipative estimate:
\begin{equation}\label{1.disest1}
\|(v(t),v'(t),w(t))\|_{\Phi_{\infty}}\le
Q(\|(v_0,v_0',w_0)\|_{\Phi_{\infty}})e^{-\beta t}+C_*
\end{equation}
for some positive constants $\beta$ and $C_*$ and monotone
function $Q$.
\end{theorem}
\begin{proof}
As we see from the proof of the previous
proposition, the only problem is to obtain a {\it dissipative}
estimate for the $L^2$-norm of $v(t)$. Indeed, if this estimate is
 obtained, analyzing the equation for the $w$-component analogously to
\eqref{1.winf}, we deduce the {\it dissipative} estimate for the
$L^\infty$-norm of $w(t)$. This, in turns, gives the dissipative
estimate for the right-hand side $h(t)$ of \eqref{1.ODE} and the
Gronwall lemma applied to inequality \eqref{1.difin} will finish
the derivation of estimate \eqref{1.disinfest}.
\par
So, we only need to obtain the dissipative estimate for the
$L^2$-norm. To this end, we multiply the first equation of
\eqref{1.eqmain} by $2(\Dt v+\eb v)$, $\eb>0$ is a small number, which will be fixed below, and {\it integrate} over
$x\in\Omega$, after that we multiply the second equation of
\eqref{1.eqmain} by $2\alpha(\Dt w+\eb w)$, integrate over $x\in\Omega$ and
take a sum of these two equations. Then, after the standard
transformations, we end up with
\begin{equation}\label{1.difin1}
\frac d{dt} Z(t)+2\alpha\|\Dt w\|^2+2((\varphi(v)-\eb)\Dt v,\Dt
v)+2\eb f(v).v+2\alpha\eb(\|\Nx w\|_{L^2}^2+\|w\|^2_{L^2})=4\alpha\eb (v,w),
\end{equation}
where
\begin{multline}\label{1.Z}
Z(t):=\|\Dt
v(t)\|^2_{L^2}+2(F(v(t)),1)+2\eb(R(v(t)),1)+2\eb(v(t),\Dt
v(t))-\\-2\alpha(v(t),w(t))+\alpha(\|\Nx
w(t)\|^2_{L^2}+(1+\eb)\|w(t)\|^2_{L^2}).
\end{multline}
We now fix $\eb>0$ so small that
$$
\eb|(R(v(t)),1)|\le 1
$$
(it is possible to do due to estimate \eqref{1.Linf}, of course,
$\eb$ will depend on the norm of the initial data). Then, due to
\eqref{1.Fest}, we have
\begin{multline}\label{1.Zest}
\beta_2[\|\Dt v\|^2_{L^2}+\|\Nx
w\|^2_{L^2}+\|w\|^2_{L^2}+(|F(v)|,1)]-C_2\le Z(t)\le\\\le \beta_1[\|\Dt v\|^2_{L^2}+\|\Nx
w\|^2_{L^2}+\|w\|^2_{L^2}+(|F(v)|,1)]+C_1,
\end{multline}
where the positive constants $C_i$ and $\beta_i$ are independent
of $\eb\to0$ and $(v,w)$. Moreover, due to the fourth assumption
of \eqref{1.assum},
\begin{equation}\label{1.Ff}
F(v)\le f(v).v+Kv^2/2.
\end{equation}
Inserting estimates \eqref{1.Zest} and \eqref{1.Ff} into
\eqref{1.difin1} and using again the third assumption of
\eqref{1.assum}, we deduce the differential inequality:
\begin{equation}\label{1.difinf3}
\Dt Z(t)+\beta\eb Z(t)\le C\eb,
\end{equation}
where $\eb$ depends on the norm of the initial data, but the
positive constants $\beta$ and $C$ are independent of $v$ and $w$.
Integrating this inequality, we arrive at
\begin{equation}\label{1.Zfin}
Z(t)\le [Z(0)-\frac C\beta]e^{-\beta\eb t}+\frac C\beta.
\end{equation}
We see that, although the rate of convergence to the absorbing
ball depends on the initial data (through the choice of $\eb>0$),
the radius of the absorbing ball is {\it independent} of $\eb$
and, consequently, is independent of the norm of the initial data.
This observation, together with estimate \eqref{1.Zest} implies
that
\begin{equation}\label{1.L2dis}
\|\Dt v(t)\|_{L^2}+\|v(t)\|_{L^2}+\|w(t)\|_{H^1}\le
Q(\|(v_0,v_0',w_0)\|_{\Phi_\infty})e^{-\gamma t}+ C_*
\end{equation}
for some positive $\gamma$ and $C$ and a monotone function $Q$
which are independent of $t$, $v$ and $w$. Thus, the desired
dissipative estimate in $L^2$ is obtained and Theorem
\ref{Th1.disinf} is proved.
\end{proof}

We now formulate several
 auxiliary results on the smoothing property for the
$w$-component and the existence of dissipative integrals in
stronger norms which will be essentially used in the next
sections.
\begin{proposition}\label{Prop1.wsmooth} Let the assumptions of Theorem
\eqref{Th1.main} hold. Then, $w(t)\in W^{2,p}(\Omega)$ and $\Dt w(t)\in
W^{2,p}(\Omega)$ for any $t>0$ and any $p<\infty$ and the
following estimate is valid:
\begin{equation}\label{1.smest}
\|w(t)\|_{W^{2,p}(\Omega)}+\|\Dt w(t)\|_{W^{2,p}(\Omega)}\le
(1+t^{-N})Q_p(\|(v_0,v_0',w_0)\|_{\Phi_{\infty}})
\end{equation}
for some positive exponent  $N$ and some monotone function $Q_p$
(depending only on $p$).
\end{proposition}
\begin{proof} Indeed, due to the smoothing property of the heat
equation, the solution $\theta(t)$ of
\begin{equation}\label{1.hom}
\Dt\theta-\Dx\theta+\theta=0,\ \ \theta\big|_{t=0}=w_0
\end{equation}
satisfies the following estimate
\begin{equation}\label{1.anal}
\|\theta(t)\|_{W^{2,p}(\Omega)}+\|\Dt\theta(t)\|_{W^{2,p}(\Omega)}\le
C_pt^{-N}\|w_0\|_{H^1}
\end{equation}
for some exponent $N$ and positive constant $C_p$ depending only
on $p$, see eg, \cite{LSU}. The remainder $z(t):=w(t)-\theta(t)$
solve the heat equation with zero initial data
$$
\Dt z-\Dx z+z=v(t),\ \ z\big|_{t=0}=0.
$$
Moreover, using estimate \eqref{1.Linf} together with the first
equation of \eqref{1.eqmain}, we conclude that
\begin{equation}\label{1.derv}
\|v(t)\|_{L^\infty(\Omega)}+\|\Dt
v(t)\|_{L^\infty(\Omega)}+\|\Dt^2 v(t)\|_{L^\infty(\Omega)}\le
Q(\|(v_0,v_0',w_0)\|_{\Phi_{\infty}}).
\end{equation}
Using that estimate together with the $W^{2,p}$-regularity
estimate for the heat equation, we arrive at
$$
\|\theta(t)\|_{W^{2,p}(\Omega)}+\|\Dt w(t)\|_{W^{2,p}(\Omega)}\le
 Q_p(\|(v_0,v_0',w_0)\|_{\Phi_{\infty}})
$$
which, together with estimate \eqref{1.anal}, finishes the proof
of the proposition.
\end{proof}

\begin{proposition}\label{Prop1.disst} Let the assumptions of Theorem
\ref{Th1.main} hold. Then, the following stronger version of
dissipative integrals exist:
\begin{equation}\label{1.dissup}
\begin{cases}
1)\ \ \ \int_1^\infty\|\Dt
w(t)\|^2_{C(\Omega)}\,dt\le Q(\|(v_0,v_0',w_0)\|_{\Phi_\infty}),\\
2)\ \ \ \int_0^T|\Dt^2 v(t,x_0)|^2+|\Dt
v(t,x_0)|^2\,dt\le  \eb T+C_\eb Q(\|(v_0,v_0',w_0)\|_{\Phi_\infty}),
\end{cases}
\end{equation}
where $\eb>0$ is arbitrary, $x_0\in\Omega$ is almost arbitrary, $C_\eb>0$ depends only on $\eb$ and $Q$ is some monotone function.
\end{proposition}
\begin{proof} We first note that, due to Lemma \ref{Prop1.wsmooth},
we may assume without loss of generality that $\Dt w(0)\in H^1$. Differentiating the equation for the $w$-component by
$t$ and denoting $z:=\Dt w$, we get
\begin{equation}\label{1.par}
\Dt z-\Dx z+z=\Dt v, \ z\big|_{t=0}=\Dt v(0).
\end{equation}
Applying the $L^2$-regularity theorem for that heat equation, we
will have
\begin{equation}\label{1.wl2}
\int_0^T\|\Dx z(s)\|_{H^2}^2\,ds\le
C\|z(0)\|_{H^1}^2+C\int_0^T\|\Dt v(s)\|^2_{L^2}\,ds,
\end{equation}
where the constant $C$ is independent of $T$. Together with
estimate \eqref{1.disint} and embedding $H^2\subset L^\infty$, it
gives the desired first estimate of \eqref{1.dissup}. Let us now
prove the second estimate of \eqref{1.dissup}. To this end, we
multiply equation \eqref{1.ODE} by $2y'$ (without integration by
$x$!). This gives
\begin{equation}\label{1.adis}
((y')^2+2F(y)-2\alpha w y)'+2\varphi(y)(y')^2=-2\alpha y\Dt
w(t,x_0).
\end{equation}
Integrating  this equality over $t\in[0,T]$,
estimating
$$
|2\alpha y \Dt w(t,x_0)|\le C\|v(t)\|_{L^\infty}\|\Dt
w(t)\|_{L^\infty}\le \eb+ \eb^{-1}\alpha^{-2}\|v(t)\|_{L^\infty}^2\|\Dt w(t)\|_{L^\infty}^2,
$$
and using the first estimate of \eqref{1.dissup} together with the strict positivity
of $\varphi$ and the fact that the $L^\infty$-norm of $v$ is under the control, we deduce that
\begin{equation}\label{1.addon}
\int_0^T[y'(t)]^2\,dt\le \eb T+ \eb^{-1}Q(\|v_0,v_0',w_0\|_{\Phi_\infty})
\end{equation}
for some (new) monotone function $Q$ which is independent of $T$.
This gives the second estimate of \eqref{1.dissup} for the term $\Dt v(t,x_0)$. Thus,
in order to finish the proof of the proposition, we only need  to estimate the term $\Dt^2 v(t,x_0)$. To this end, we differentiate the first equation of \eqref{1.eqmain} by $t$ and denote $q(t):=\Dt v(t,x_0)$. Then, we get
$$
q''+\varphi(y)q'+\varphi'(v)q^2+f'(y)q=\alpha\Dt w(t,x_0).
$$
Multiplying this equation by $2q'$, integrating by time and using that the $L^\infty$-norms of $v$, $\Dt v$ and $\Dt^2 v$ are under the control, we arrive at
$$
\int_0^T[q'(t)]^2\,dt\le Q(\|v_0,v_0',w_0\|_{\Phi_\infty})\(1+\int_0^T[y'(t)]^2+\|\Dt w(t)\|^2_{L^\infty}\,dt\),
$$
for some monotone function $Q$ which is independent of $T$. Inserting estimate \eqref{1.addon} and the first estimate of \eqref{1.dissup} to that inequality (and scaling the parameter $\eb$ if necessary), we obtain the desired control for the integral of $\Dt^2 v$ and finish the proof of the proposition.
\end{proof}

We conclude this section by showing that, if the initial data
$(v_0,v_0',w_0)$ is smooth, the solution  $(v(t),w(t))$ remains smooth
for all $t$.
\begin{proposition}\label{Prop1.smooth} Let the assumptions of
Theorem \eqref{Th1.main} hold. Assume, in addition, that
\begin{equation}\label{1.idsm}
(v_0,v_0',w_0)\in W^{1,\infty}(\Omega).
\end{equation}
Then, the solution $(v(t),\Dt v(t),w(t))\in W^{1,\infty}(\Omega)$
for any $t\ge0$ and the following estimate holds:
\begin{equation}\label{1.divreg}
\|v(t)\|_{W^{1,\infty}}+\|\Dt
v(t)\|_{W^{1,\infty}}+\|w(t)\|_{W^{1,\infty}}\le
C\|(v_0,v_0',w_0)\|_{[W^{1,\infty}]^3}e^{Kt}
\end{equation}
for some positive constants $C$ and $K$
(which depend on the $L^\infty$-norms of the initial data).
\end{proposition}
\begin{proof}
  The desired estimate for the $w$-component is
factually obtained in Proposition \ref{Prop1.wsmooth}, thus, we
only need to estimate the $v$ component. Let $x_1$ and $x_2$ be
two arbitrary points of $\Omega$ and let
$z(t):=v(t,x_1)-v(t,x_2)$. Then, this function satisfies the
following ODE
\begin{multline}\label{1.ODEd}
z''(t)+\varphi(v(t,x_1))z'(t)+[\varphi(v(t,x_1))-\varphi(v(t,x_2))]\Dt
v(t,x_2)+\\+[f(v(t,x_1))-f(v(t,x_2))]=h_{x_1,x_2}(t):=\alpha(w(t,x_1)-w(t,x_2)).
\end{multline}
Multiplying this equation by $\Dt z(t)$ and arguing exactly as in
\eqref{1.difinf}, we arrive at
\begin{equation}\label{1.Lip}
|z(t)|^2+|z'(t)|^2\le C(|z(0)|^2+|
z'(0)|^2)e^{Kt}+C\int_0^te^{K(t-s)}|h_{x_1,x_2}(s)|^2\,ds,
\end{equation}
where the constants $C$ and $K$ depends on the $L^\infty$-norm of
the solution. Furthermore, since the $W^{1,\infty}$-estimate for
the $w$-component is already obtained, we have
$$
\sup_{x_1,x_2\in\Omega}\frac1{|x_1-x_2|^2}\int_0^te^{K(t-s)}|h_{x_1,x_2}(s)|^2\,ds\le
C_1\int_0^te^{K(t-s)}\|w(s)\|^2_{W^{1,\infty}}\,ds\le Ce^{Kt}.
$$
Dividing finally inequality \eqref{1.Lip} by $|x_1-x_2|^2$ and
taking the supremum over $x_1,x_2\in\Omega$ from the both parts of
the obtained inequality, we obtain the desired estimate for the
$W^{1,\infty}$-norms of $v$ and $\Dt v$ and finish the proof of
the proposition.
\end{proof}

\begin{remark}\label{Rem1.last} Arguing analogously, one can show
that if the initial data belong to $C^k$, the solution will be of
class $C^k$ for every $t\ge0$. Thus, the blow up in finite time of the  higher norms
cannot occur. However, there is a principal difference between
estimates \eqref{1.disest1} for the $L^\infty$-norm and estimate
\eqref{1.divreg} for the $W^{1,\infty}$-norm of the solution.
Indeed, the first estimate is dissipative and shows that the
$L^\infty$-norm of the solution cannot grow and even gives the
absorbing ball in that norm. In contrast to that,  the
$W^{1,\infty}$-norm, a priori, may grow exponentially and, in this
sense, the solution may become "less and less regular" as
$t\to\infty$ (i.e., it may tend to a discontinuous limit). As we
will see in the next sections, the answer on the question whether
or not it really  happens depends in a crucial way on the
monotonicity of the nonlinearity $f$.
\end{remark}

\section{The monotone case: asymptotic compactness and regular
attractor}\label{s2}

According to the results of the previous section, equation
\eqref{1.eqmain} is uniquely solvable in the phase space
$\Phi_\infty$ and the solution operators
\begin{equation}\label{2.sem}
S(t)(v_0,v_0',w_0):=(v(t),\Dt v(t),w(t))
\end{equation}
generate a dissipative semigroup in $\Phi_\infty$. The aim of this
section is to study the long-time behavior of solutions as
$t\to\infty$ in the particular case where the nonlinearity $f$ is
strictly monotone:
\begin{equation}\label{2.mon}
f'(v)\ge\kappa_0>0.
\end{equation}
As we will see, in that case, the associated semigroup is
asymptotically compact and possesses a smooth global attractor $\Cal A$ in
$\Phi_\infty$. Moreover, due to the Lyapunov functional, this
attractor can be described as a finite union of finite-dimensional
unstable manifolds. Our proof of the asymptotic compactness is
based on the following lemma which can be considered as a
refinement of estimate \eqref{1.divreg}.

\begin{lemma}\label{Lem2.smo} Let the assumptions of Theorem
\ref{1.eqmain} hold and let, in addition, \eqref{2.mon} be
satisfied. Let us also introduce, for any $h>0$ the following (semi)norm
on the space $L^\infty(\Omega)$:
\begin{equation}\label{2.lipn}
\|v\|_{W^{1,\infty}_h}:=\sup_{x_1,x_2\in\Omega,\, |x_1-x_2|\ge h}\frac{|v(x_1)-v(x_2)|}{|x_1-x_2|}
\end{equation}
(being pedants, we would  write $\operatorname{esssup}$ instead of
$\sup$). Then, every solution $(v(t),w(t))$ of problem
\eqref{1.eqmain} satisfies
\begin{equation}\label{2.smest}
\|v(t)\|_{W^{1,\infty}_h}+\|\Dt v(t)\|_{W^{1,\infty}_h}\le
\frac{C_1}h e^{-\beta t}+ C_2,
\end{equation}
where the positive constants $\beta$, $C_i$ depend on the $\Phi_\infty$-norms of
the initial data, but are independent of $t$ and $h\to0$.
\end{lemma}
\begin{proof} Analogously to the proof of Proposition
\ref{Prop1.smooth}, we introduce a function
$z(t):=v(t,x_1)-v(t,x_2)$ which solves equation \eqref{1.ODEd}.
But, using the monotonicity assumption \eqref{2.mon} and the dissipation integrals \eqref{1.dissup}, we are now
able to suppress the exponential divergence in estimate
\eqref{1.Lip}. To this end, we multiply equation \eqref{1.ODEd} by
$z'$ and transform the term containing the nonlinearity $f$ as
follows:
\begin{equation}\label{2.f1}
[f(v(t,x_1))-f(v(t,x_2))]z'=1/2[l(t)z^2]'-1/2 l'(t)z^2,
\end{equation}
where $l(t):=\int_0^1f'(sv(t,x_1)+(1-s)v(t,x_2))\,ds\ge\kappa_0>0$ and its
derivative can be estimated as follows:
\begin{equation}\label{2.der}
|l'(t)|\le C(|\Dt v(t,x_1)|+|\Dt v(t,x_2)|),
\end{equation}
where the constant $C$ depends on the $L^\infty$-norm of the
initial data, but is independent of $t$ and $x_i$. Then,
using the positivity of $\varphi$ and the $L^\infty$-bounds for
$v$, we get
$$
1/2((z')^2+lz^2)'+\gamma(z')^2\le C|h_{x_1,x_2}|^2+C(|\Dt
v(t,x_1)|+|\Dt v(t,x_2)|)z^2
$$
for some positive constants $\gamma$ and $C$. Multiplying now
equation \eqref{1.ODEd} by $\eb z$ (where $\eb>0$ is a
sufficiently small positive number) and taking a sum with the
above inequality, we infer
\begin{multline}\label{2.dif1}
1/2((z')^2+lz^2+2\eb zz')'+(\gamma-\eb)(z')^2+\eb l z^2\le\\\le
C|h_{x_1,x_2}|^2+C(|\Dt v(t,x_1)|+|\Dt v(t,x_2)|)(z^2+(z')^2).
\end{multline}
Let now
$$
\Cal L_z(t):=(z')^2+lz^2+2\eb zz'.
$$
Then, since $l(t)\ge\kappa_0>0$, we may fix $\eb>0$ to be small
enough that
\begin{equation}\label{2.L}
\kappa(z^2+(z')^2)\le\Cal L_z(t)\le \kappa_1(z^2+(z')^2)
\end{equation}
for some positive $\kappa$ and $\kappa_1$. This, inequality,
together with the evident estimate $|x|\le \beta+\beta^{-1}x^2$ allows to
transform \eqref{2.dif1} to
\begin{equation}\label{2.difin}
\frac d{dt}\Cal L_z(t)+(\gamma-C(|\Dt v(t,x_1)|^2+|\Dt
v(t,x_2)|^2))\Cal L_z(t)\le C|h_{x_1,x_2}(t)|^2.
\end{equation}
Applying the Gronwall inequality to this relation, we arrive at
$$
\Cal L_z(t)\le \Cal L_z(0)e^{-\int_0^T K(s)\,ds}+C\int_0^Te^{-\int_t^TK(s)\,ds}|h_{x_1,x_2}(t)|^2\,dt
$$
with $K(s):=\gamma-C(\Dt v(t,x_1)|^2+|\Dt
v(t,x_2)|^2)$. Using the
dissipation integral \eqref{1.dissup} with $\eb:=\gamma/2$, we see that
$$
\int^T_tK(s)\,ds\ge \gamma(T-t)/2-C,
$$
where the constant $C$ depend on the $\Phi_\infty$-norm of the initial data. This estimate,
 together with the bounds
\eqref{2.L}, give the non-divergent analogue of
estimate \eqref{1.Lip}:
\begin{equation}\label{2.divc}
|z(t)|^2+|z'(t)|^2\le C(|z(0)|^2+|z'(0)|^2)e^{-\gamma t/2}+
C\int_0^te^{-\gamma(t-s)/2}|h_{x_1,x_2}(s)|^2\,ds
\end{equation}
for some positive $C$ and $\gamma$ depending only on the
$\Phi_\infty$-norm of the solution.
\par
In order to deduce the desired estimate \eqref{2.lipn} from
\eqref{2.divc}, we note that, due to Proposition
\ref{Prop1.wsmooth}, we may assume without loss of generality that
$\|w(t)\|_{W^{1,\infty}}\le C$ for all $t\ge0$ and, consequently,
$$
\sup_{|x_1-x_2|\ge
h}\frac1{|x_1-x_2|^2}\int_0^te^{-\gamma(t-s)/2}|h_{x_1,x_2}(s)|^2\,ds\le
\alpha\int_0^t e^{-\gamma(t-s)/2}\|w(s)\|_{W^{1,\infty}}^2\,ds\le
C_1.
$$
Moreover, obviously,
$$
\|v\|_{W^{1,\infty}_h}\le \frac{2\|v\|_{L^\infty}}h.
$$
Dividing now inequality \eqref{2.divc} by $|x_1-x_2|^2$ and taking
the supremum over all $x_i\in\Omega$, $|x_1-x_2|\ge h$, we deduce
the desired estimate \eqref{2.lipn}. Lemma \ref{Lem2.smo} is
proved.
\end{proof}

Our next step is to verify the existence of a global attractor
$\Cal A$ for semigroup \eqref{2.sem} associated with problem
\eqref{1.eqmain}. We recall that, by definition, the global
attractor $\Cal A$ should satisfy the following properties:
\par
1) $\Cal A$ is compact in $\Phi_\infty$;
\par
2) $\Cal A$ is strictly invariant: $S(t)\Cal A=\Cal A$;
\par
3) $\Cal A$ attracts the images of all bounded sets as
$t\to\infty$, i.e., for any bounded set $B\subset\Phi_\infty$ and
any neighbourhood $\Cal O(\Cal A)$ of $\Cal A$ in $\Phi_\infty$,
there exists time $T=T(B,\Cal A)$ such that
$$
S(t)B\subset\Cal O(\Cal A), \ \ \ \text{for $t\ge T$}.
$$
We also recall that the attraction property can be also
reformulated in terms of the non-symmetric Hausdorff distance
between sets in $\Phi_\infty$:
\begin{equation}\label{2.H}
\lim_{t\to\infty}\dist(S(t)B,\Cal A)=0,
\end{equation}
see eg, \cite{bv} for the details.
\begin{theorem}\label{Th2.attr} Let the assumptions of Lemma
\ref{Lem2.smo} hold. Then, the semigroup $S(t)$ associated with
problem \eqref{1.eqmain} possesses a global attractor $\Cal A$ in
the phase space $\Phi_\infty$. This attractor is bounded in
$[W^{1,\infty}(\Omega)]^3$ and has the following structure:
\begin{equation}\label{2.ker}
\Cal A=\Cal K\big|_{t=0},
\end{equation}
where $\Cal K\subset C_b(\R,\Phi_\infty)$ is a set of all
solutions of problem \eqref{1.eqmain} which are defined for all
$t\in\R$ and are globally bounded.
\end{theorem}
\begin{proof} In order to deduce the existence of a global
attractor from Lemma \ref{Lem2.smo}, we will use the so-called
Kuratowski measure of non-compactness. Recall that, by definition,
the Kuratowski measure of non-compactness $\alpha(B)$ of a set $B$ is
infinum of all $r>0$ for which it can be covered by a finite
number of $r$-balls, see e.g., \cite{hale} for details. To be more
precise, we need the following lemma.

\begin{lemma}\label{Lem2.kur} Let
$$
B:=\{v\in L^\infty(\Omega),\ \
\|v\|_{L^\infty}+\|v\|_{W^{1,\infty}_h}\le R\}
$$
for some $R$ and $h>0$. Then, its Kuratowski measure of
non-compactness of the set $B$ can be estimated as follows:
\begin{equation}\label{2.kest}
\alpha(B)\le Rh.
\end{equation}
\end{lemma}
\begin{proof} Let $\Cal S_h$ be the standard averaging operator
$$
(\Cal S_hv)(x):=\int_{\R^3} D_h(x,y)v(y)\,ds,
$$
where the smooth non-negative kernels $D_h(x,z)$ are such that
\begin{equation}\begin{cases}
1.\ \ \supp D_h(x,\cdot)\subset\{z\in\Omega,\  |z-x|\le h\}\\
2.\ \ \int_{\R^3}D_h(x,y)\,dy\equiv1,\\
3.\ \ |D_h(x,y)|+|\Nx D_h(x,y)|\le C_h,\ \ x,y\in\R^3
\end{cases}
\end{equation}
(since $\Omega$ is assumed to be smooth, such kernels exist).
\par
Let also $B_h:=\Cal S_h(B)$. Then, on the one hand, the
set $B_h$ consists of smooth functions and, in particular, is
bounded in $C^1(\Omega)$. By the Arzela-Ascoli theorem, it means
that $B_h$ is compact in $L^\infty(\Omega)$.
\par
On the other hand,
$$
|(\Cal S_hv)(x)-v(x)|\le \int_{\R^3} D_h(x,y)|v(y)-v(x)|\,dy\le
\|v\|_{W^{1,\infty}_h}h\int_{\R^3}D_h(x,y)\,dy\le Rh.
$$
Thus, $B\subset B_h+Rh$ and $B_h$ is compact. This gives estimate
\eqref{2.kest} and finishes the proof of the lemma.
\end{proof}
We are now ready to finish the proof of the theorem. Indeed,
due to Proposition \ref{Prop1.wsmooth}, we know that the
$w$-component is bounded in $W^{2,p}(\Omega)$ for every $t>0$ and,
consequently, the $w$-component of $S(t)B$ is precompact in $L^\infty\cap H^1$ for any
bounded set $B$. So, the Kuratowski measure of
non-compactness for $S(t)B$ is determined by the $v$-component
only. Moreover,  Lemma \ref{Lem2.smo} guarantees, that
$$
v(t),\Dt v(t)\subset\{u\in L^\infty(\Omega),\
\|u\|_{L^\infty}+\|u\|_{W^{1,\infty}_h}\le R\}
$$
if $t\ge T(h)$ is large enough (but $R$ is independent of $h$).
This, gives that
\begin{equation}\label{2.kurlim}
\lim_{t\to\infty}\alpha(S(t)B)=0
\end{equation}
for any bounded set $B$.
\par
 Since the semigroup $S(t)$ is Lipschitz
continuous with respect to the initial data (see Theorem
\ref{Th1.main}) and dissipative (see Theorem \ref{Th1.disinf}),
the convergence of the Kuratowski measure \eqref{2.kurlim} to zero
implies the asymptotic compactness of the semigroup $S(t)$ and the
existence of a global attractor $\Cal A$, see \cite{hale}. The
structure \eqref{2.ker} of the attractor is also a corollary of
that abstract theorem and the fact that $\Cal A$ is bounded in
$W^{1,\infty}$ follows from estimate \eqref{2.lipn} (together with
the fact that the constant $C_2$ is independent of $h$). Thus,
Theorem \ref{Th2.attr} is proved.
\end{proof}

Our next task is to establish the regular structure of the
attractor $\Cal A$ provided by the Lyapunov functional. To this
end, we need to make some preparations. As a first step, we
establish the differentiability of the semigroup $S(t)$ with
respect to the initial data.

\begin{proposition}\label{Prop2.dif} Let the assumptions of
Theorem \ref{Th1.main} hold. Then, the associated semigroup $S(t)$
is Frechet differentiable  with respect to the initial data
for every fixed $t$ and it's Frechet derivative $D_\xi
[S(t)\xi]\in \Cal L(\Phi_\infty,\Phi_\infty)$ is Lipschitz
continuous with respect to the initial data $\xi\in\Phi_\infty$
and the following estimate holds for every bounded set $B\subset \Phi_\infty$:
\begin{equation}\label{2.smooth}
\|S(t)\|_{C^{1,1}(B,\Phi_\infty)}\le Ce^{2Kt},
\end{equation}
where the constants $C$ and $K$ depend only on the radius of $B$.
\end{proposition}
The proof of this proposition is straightforward and standard, so,
in order to avoid the technicalities,
we rest it to the reader.
\par
At the next step, we need to study the equilibria of
problem \eqref{1.eqmain}.

\begin{proposition}\label{Prop2.eq} Let the assumptions of Theorem
\ref{Th2.attr} hold. Then, any equilibrium $(v_0,w_0)\in\Cal R$
(the set of all equilibria) of problem \eqref{1.eqmain} solves the
following semilinear elliptic equation:
\begin{equation}\label{2.eleq}
-\Dx w+w=f^{-1}(\alpha w),\ \  \partial_n\big|_{\partial\Omega}=0,
\ \ v=(-\Dx+1)w
\end{equation}
($f^{-1}$ exists since $f$ is now assumed to be monotone).
Moreover, the equilibrium $(v_0,w_0)$ is hyperbolic if and only if
$w_0$ is hyperbolic as a solution of \eqref{2.eleq}, i.e., if the
equation
\begin{equation}\label{2.nhyp}
-\Dx \theta+\theta=[f^{-1}]'(\alpha w_0)\alpha \theta
\end{equation}
has only trivial solution $\theta=0$. In particular, for generic $f$,
all of the equilibria $(v_0,w_0)\in\Cal R$ are hyperbolic and
$\Cal R$ is finite.
\end{proposition}
\begin{proof} Indeed, the equations on equilibria for problem
\eqref{1.eqmain}
$$
f(v)=\alpha w,\ \ -\Dx w+w=v
$$
are equivalent to \eqref{2.eleq}. Let us verify the assertion on
hyperbolicity. Indeed, the asymptotic compactness of the semigroup
$S(t)$ implies in a standard way that the essential spectrum of
the operator $D_\xi S(1)$ lies strictly inside of the unit circle.
Thus, only eigenvalues of finite multiplicity are possible on the
unit circle. Any such eigenvalue generates a {\it time-periodic}
solution $(z,\theta)$ of the associated equation of variations
\begin{equation}\label{2.ev}
\begin{cases}
\Dt^2 z+\varphi(v_0)\Dt z+f'(v_0)z=\alpha \theta,\\
\Dt \theta-\Dx\theta+\theta=z.
\end{cases}
\end{equation}
However, analogously to the nonlinear problem \eqref{1.eqmain},
the linearized problem \eqref{2.ev} possesses a global Lyapunov
function (in order to find it, one needs to multiply the first and
the second equations by $\Dt z$ and $\alpha\Dt \theta$
respectively, take a sum and integrate over $\Omega$). Thus,
every periodic solution of that linearized problem must be an
equilibrium:
$$
f'(v_0)z=\alpha\theta,\ \ z=-\Dx\theta+\theta
$$
and, consequently, $z$ must solve \eqref{2.nhyp}. Vise versa, any
nontrivial solution $z$ of \eqref{2.nhyp} generates a non-trivial
equilibrium of \eqref{2.ev} by setting $z=-\Dx\theta+\theta$.
\par
Finally, the last assertion that generically $\Cal R$ is finite
and all of the equilibria are hyperbolic is a standard corollary
of the Sard theorem, see eg, \cite{bv}. Proposition \ref{Prop2.eq}
is proved.
\end{proof}
Thus, we will assume from now on that all of the equilibria
$(v_0,w_0)\in\Cal R$ are hyperbolic (which automatically implies
that $\Cal R$ is finite). Furthermore, for any
$\xi_0:=(v_0,w_0)\in\Cal R$, we define the associated unstable
set $\Cal M_{\xi_0}^+$ by the usual expression
\begin{multline}\label{2.uns}
\Cal M^+_{\xi_0}:=\{(v_0,v_0',w_0)\in\Phi_\infty,\ \exists
(v(t),w(t))\,\ \ \text{ which solves \eqref{1.eqmain} for $t\le0$
such that}\\ (v(0),\Dt v(0),w(0))=(v_0,v_0',w_0)\ \text{ and }\
\lim_{t\to-\infty}(v(t),\Dt v(t),w(t))=(v_0,v_0',w_0)\}.
\end{multline}
In other words, the unstable set $\Cal M^+_{\xi_0}$ consists of
all complete trajectories of \eqref{1.eqmain} which stabilize to
$\xi_0$ as $t\to-\infty$.
\par
It is well known (see e.g., \cite{bv}) that, for hyperbolic equilibrium $\xi_0\in\Cal
R$, the set $\Cal M^+_{\xi_0}$ is {\it locally} (near $\xi_0$) a
finite-dimensional submanifold of $\Phi_\infty$ and its dimension
equals to the instability index of $\xi_0$. But, in order to prove
that the whole $\Cal M^+_{\xi_0}$ is a submanifold of $\Phi_\infty$,
one needs the semigroup $S(t)$ to be injective (in other words, problem \eqref{1.eqmain} should possess the so-called backward uniqueness property, see again
\cite{bv}).

\begin{proposition}\label{Prop2.inj} Let the assumptions of
Theorem \ref{Th1.main} hold. Then, the semigroup $S(t)$ associated
with equation \eqref{1.eqmain} is injective, i.e.,
the equality $S(T)\xi_1=S(T)\xi_2$, for some $T>0$, implies that
$\xi_1=\xi_2$.
\end{proposition}
\begin{proof} Indeed, let $(v_1(t),w_1(t))$ and $(v_2(t),w_2(t))$
be two solutions of problem \eqref{1.eqmain} and let $(z(t),w(t))$
be their difference. Then, these functions solve \eqref{1.eqdif}.
Let us rewrite this equation in the form
\begin{equation}\label{2.vec}
\Dt\xi+\Cal B\xi=\Cal P(t)\xi,
\end{equation}
where $\xi(t):=(z(t),\Dt z(t),\theta(t))$,
\begin{equation*}\label{2.huge}
\Cal B=\(\begin{matrix} 1&0&0\\0&1&0\\0&0&-\Dx+1\end{matrix}\),\ \
\Cal P(t):=\(\begin{matrix}1&1&0\\-[l_\varphi(t)\Dt
v_2(t)+l_f(t)]&1-\varphi(v_1(t))& \alpha\\1&0&0\end{matrix}\)
\end{equation*}
and
$$
l_\varphi(t):=\int_0^1\varphi'(sv_1(t)+(1-s)v_2(t))\,ds,\ \
l_f(t):=\int_0^1\varphi'(sv_1(t)+(1-s)v_2(t))\,ds.
$$
Let us consider equation \eqref{2.vec} in a Hilbert space
$H:=[L^2(\Omega)]^3$. Then, obviously, $\Cal B$ is a positive
self-adjoint (unbounded) operator in $H$ and the operator $\Cal P(t)$
is uniformly bounded for all $t\ge0$. Then, the classical backward
uniqueness theorem of Agmon and Nirenberg (see \cite{an}) is
applicable here and, consequently, $\xi(T)=0$ implies that
$\xi(0)=0$. Proposition \ref{Prop2.inj} is proved.
\end{proof}

\begin{corollary}\label{Cor2.man} Let the assumptions of Theorem
\ref{Th2.attr} hold. Then, for any hyperbolic equilibrium
$\xi_0:=(v_0,w_0)$ of problem \eqref{1.eqmain}, the associated
unstable manifold $\Cal M^+_{\xi_0}$ is a finite-dimensional
$C^1$-submanifold of $\Phi_\infty$ diffeomorphic to $\R^N$, where
$N$ is the instability index of the equilibrium~$\xi_0$.
\end{corollary}
This result is a standard corollary of the existence of a Lyapunov
function, injectivity and smoothness of the semigroup $S(t)$, see
\cite{bv}.
\par
We are now ready to formulate a theorem on the regular structure
of the attractor $\Cal A$ which can be considered  as the main
result of the section.
\begin{theorem}\label{Th2.areg} Let the assumptions of Theorem
\ref{Th2.attr} hold and let, in addition, all of the equilibria
$\xi_0\in\Cal R$ be hyperbolic. Then,
\par
1) any non-equilibrium complete trajectory $\xi(t)$, $t\in\R$ of the semigroup
$S(t)$ belonging to the attractor  is a heteroclinic orbit between
two different equilibria $\xi_-$ and $\xi_+$:
\begin{equation}\label{2.hetero}
\lim_{t\pm\infty}\xi(t)=\xi_\pm,
\end{equation}
where $\xi_-\ne\xi_+$;
\par
2) The attractor $\Cal A$ is a finite union of finite-dimensional
submanifolds of $\Phi_\infty$:
\begin{equation}\label{2.subma}
\Cal A=\cup_{\xi_0\in\Cal R}\Cal M^+_{\xi_0},
\end{equation}
where $\Cal M^+_{\xi_0}$ is an unstable manifold of $\xi_0\in\Cal R$;
\par
3) The rate of attraction to $\Cal A$ is exponential, i.e., for
any bounded set $B\subset\Phi_\infty$,
\begin{equation}\label{2.eattr}
\dist(S(t)B,\Cal A)\le Q(\|B\|_{\Phi_\infty})e^{-\gamma t}
\end{equation}
for some monotone function $Q$ and positive constant $\gamma$.
\end{theorem}
Indeed, this theorem follows from the abstract result of Babin and
Vishik on  regular attractors, see \cite{bv} (see also \cite{vz})
and Propositions \ref{Prop1.Lyap}, \ref{Prop2.dif} and
\ref{Prop2.inj}.

\begin{remark}\label{Rem2.strange} Theorem \ref{Th2.areg} shows
that the long-time behavior of solutions of problem
\eqref{1.eqmain} is "extremely good" if the monotonicity
assumption
\eqref{2.mon} holds. As we will see in the next section, this
behavior is "extremely bad" if the monotonicity assumption is
essentially violated.
\end{remark}

\section{The non-monotone $f$: stabilization for the case of weak
coupling}\label{s3}
The aim of this section is to understand how the asymptotic behavior of
\eqref{1.eqmain} may look like when the monotonicity assumption
\eqref{2.mon} is violated. To this end, we will consider below the
case of small coupling constant $\alpha$, where the dynamics is,
in a sense, determined by the limit ODE
\begin{equation}\label{3.ODE}
y''(t) +\varphi(y(t))y'(t)+f(y(t))=0.
\end{equation}
In particular, in that limit case, the value of $v(t,x)$ at
$x=x_0$ is determined by the value of $(v_0(x),v_0'(x))$ at
$x=x_0$ only and, consequently $v(t,x)$ and $v(t,y)$ evolve
independently if $x\ne y$. Thus, if \eqref{3.ODE} has more than
one equilibrium, the most part of trajectories will tend to a
{\it discontinuous} equilibria, no matter whether or not the
initial data $(v_0,v_0')$ is continuous. As we will see, the same
property preserves for the case of small positive coupling
constant $\alpha$.
\par
To be more precise, we assume that the limit equation
\eqref{3.ODE} possesses a regular attractor in $\R^2$, i.e., that
\begin{equation}\label{3.eqhyp}
f'(u_0)\ne0,\ \ \text{ for all $u_0$ such that } \ f(u_0)=0
\end{equation}
(since the existence of a global Lyapunov function and
dissipativity are immediate, only the hyperbolicity assumption on
the equilibria should be postulated).
\par
We start our exposition by verifying that the dissipative estimate
\eqref{1.disest1} is uniform with respect to $\alpha\to0$.

\begin{proposition}\label{Prop3.undis} Let the assumptions of
Theorem \ref{Th1.disinf} hold. Then the positive constants $\beta$
and $C_*$ and monotone function $Q$ in the dissipative estimate \eqref{1.disest1} are independent of
$\alpha\to0$.
\end{proposition}
\begin{proof} In order to verify this assertion, we need to check
that the most part of estimates of Section \ref{s1} are uniform
with respect to $\alpha\to0$. We start with estimate
\eqref{1.l2bound}. From the first point of view (based on the form
of the global Lyapunov function \eqref{1.eq}), one may guess that
it is non-uniform and only $\alpha \|w(t)\|^2_{L^2}$ is uniformly
bounded. However, the Lyapunov function gives the uniform bound
for the $L^2$-norm of the $v$-component. The standard $L^2$-estimate for second equation
\begin{equation}\label{3.mains}
\Dt w-\Dx w+w=v(t)
\end{equation}
of \eqref{1.eqmain} gives after that the uniform $L^2$ and
$L^\infty$-bounds for the $w$-component.
\par
Thus, the uniformity as $\alpha\to0$ is verified for estimates
\eqref{1.l2bound} and \eqref{1.winf}. The uniformity of the
$L^\infty$-bound \eqref{1.Linf}  follows  from  \eqref{1.winf} exactly as in Proposition
\ref{Prop1.Linfbound}.
\par
So, it only remains to check the dissipative estimate
\eqref{1.disest1} and, following the proof of Theorem
\ref{Th1.disinf}, we see that only the uniformity of the
$L^2$-estimate \eqref{1.L2dis} is necessary.
\par
Analogously to \eqref{1.l2bound} the function $Z(t)$ (defined by \eqref{1.Z}) can be
estimated as follows
\begin{multline}\label{3.Zest}
\beta_2[\|\Dt v\|^2_{L^2}+
\alpha(\|\Nx w\|^2_{L^2}+\|w\|^2_{L^2})+(|F(v)|,1)]-C_2\le Z(t)\le\\\le \beta_1[\|\Dt v\|^2_{L^2}+
\alpha(\|\Nx w\|^2_{L^2}+\|w\|^2_{L^2})+(|F(v)|,1)]+C_1,
\end{multline}
where $\beta_i$ and $C_i$ are now uniform with respect to
$\alpha\to0$. By this reason, estimates \eqref{1.difinf3} and
\eqref{1.Zfin} do not give immediately the uniform analogue of
\eqref{1.L2dis}, but only the uniform dissipative estimate for the
$L^2$-norm of $v(t)$. Combining after that this estimate with the
usual $L^2$-estimate for the heat equation \eqref{3.mains}, we
verify that estimate \eqref{1.L2dis} is indeed uniform as
$\alpha\to0$. Exactly as in Theorem \ref{Th1.disinf}, this gives
the uniformity of estimate \eqref{1.disest1} and finishes the
proof of the proposition.
\end{proof}
Thus, due to the previous proposition, the radius of the absorbing
ball in $\Phi_{\infty}$ for problem \eqref{1.eqmain} is uniform with
respect to $\alpha\to0$. In particular, the $\Phi_\infty$-norm of
any equilibria of that problem is uniformly bounded. Denoting the
set of equilibria for problem \eqref{1.eqmain} by $\Cal R_\alpha$, we may conclude
that
\begin{equation}\label{3.eqbou}
\|\Cal R_\alpha\|_{\Phi_\infty}\le C,
\end{equation}
where the constant $C$ is independent of $\alpha$.
\par
This observation together with the hyperbolicity assumption
\eqref{3.eqhyp} allow to give a complete description of the
equilibria set $\Cal R_\alpha$ if $\alpha>0$ is small enough.

\begin{proposition}\label{Prop3.eqstr} Let the assumptions of
Theorem \ref{Th1.main} hold and let, in addition, the limit
hyperbolicity assumption \eqref{3.eqhyp} be valid.
Denote these hyperbolic equilibria by $\{u_1,\cdots,u_N\}$.
 Then, there
exists $\alpha_0>0$ such that, for every $\alpha\le\alpha_0$ and
every partition
\begin{equation}\label{3.part}
\Omega=\Omega_1\cup\Omega_2\cup\cdots\cup\Omega_N
\end{equation}
on disjoint measurable sets: $\Omega_i\cap\Omega_j=\varnothing$
for $i\ne j$, there exists a unique equilibrium $(v_0,0,w_0)\in\Phi_\infty$ of
problem \eqref{1.eqmain} such that
\begin{equation}\label{3.eqstr}
v_0=\tilde v_0+\theta,\ \ \tilde v_0(x):=\sum_{i=1}^N
u_i\chi_{\Omega_i}(x),\ \ \|\theta\|_{L^\infty}\le C\alpha,
\end{equation}
where $\chi_V(x)$ is a characteristic function of the set $V$ and
the constant $C$ is independent of $\alpha$. Moreover, every
equilibrium $(v_0,0,w_0)\in\Phi_\infty$ can be presented in such
form.
\end{proposition}
\begin{proof} Indeed, in order to find the equilibrium, we need to solve
$$
f(v_0)=\alpha w_0,\ \ \Dx w_0-w_0=v_0
$$
which we rewrite in the form of a single equation on $v_0$ in
$L^\infty(\Omega)$:
\begin{equation}\label{3.eqeq}
f(v_0)=\alpha(-\Dx+1)^{-1}v_0.
\end{equation}
We note that the function
$$
F(v,\alpha):=f(v)-\alpha(-\Dx+1)^{-1}v
$$
belongs to $C^2(L^\infty(\Omega)\times\R,L^\infty(\Omega))$.
Moreover, its derivative
$$
D_vF(\tilde v_0,0):=f'(\tilde v_0)
$$
is invertible in $L^\infty$ (due to the hyperbolicity assumption \eqref{3.eqhyp} and the
norm of the inverse operator is uniformly bounded with respect to
the choice of a partition. In addition, $F(\tilde v_0,0)=0$. Thus, the existence and uniqueness of
the equilibrium $v_0$ in a small neighbourhood of $\tilde v_0$ if
$\alpha$ is small follows from the implicit function theorem.
\par
Let us now verify that any equilibrium $(v_0,w_0)$ can be
presented in that form. Indeed, let $(v_0,0,w_0)\in\Phi_\infty$ be an arbitrary equilibrium. Then,
according to \eqref{3.eqbou}, $\|w_0\|_{L^\infty}\le C$ where $C$
is independent of $\alpha$. Therefore,
$$
|f(v_0(x))|\le C\alpha,\ \ \ x\in\Omega.
$$
Since all of the roots $f(z)=0$ are hyperbolic, for sufficiently
small $\alpha$, we conclude from here that
\begin{equation}\label{3.close}
v_0(x)\in\Cal O_{C\alpha}(u_{k(x)}),\ \ x\in\Omega
\end{equation}
for some root $u_{k(x)}$ of $f(z)=0$. Fixing now
\begin{equation}\label{3.discont}
\Omega_i:=\{x\in\Omega,\ k(x)=i\},
\end{equation}
we see that the equilibrium $(v_0,w_0)$ indeed has the form of
\eqref{3.eqstr} and the proposition is proved.
\end{proof}

\begin{remark}\label{Rem3.equ} We see that, in contrast to the
case of monotone $f$, we now have the {\it uncountable} number of
different equilibria (all of them are hyperbolic in $\Phi_\infty$)
most of which are {\it discontinuous} (we have only finite number
of continuous equilibria associated with trivial partitions of
$\Omega$). Moreover, using the explicit description given in the previous proposition, it is not difficult to show that the set
$\Cal R_\alpha$ is not compact in the strong topology of the space
$\Phi_\infty$ and is not closed in the weak-star topology of that
space. By this reason, the possibility to apply the strong/weak global
attractor theory to that problem seems very problematic. However,
as the next theorem shows, any trajectory $(v(t),w(t))$ still
converges to one of the equilibrium from $\Cal R_\alpha$ as
$t\to\infty$.
\end{remark}

\begin{theorem}\label{Th3.conv} Let the assumptions of Proposition
\ref{Prop3.eqstr} hold. Then, there exists $\alpha_0>0$ such that,
for every $\alpha\le\alpha_0$ every trajectory $(v(t),\Dt
v(t),w(t))$ of problem \eqref{1.eqmain} stabilizes as $t\to\infty$
to some equilibrium $(\bar v,0,\bar w)\in\Cal R_\alpha$ in the topology
of $L^p(\Omega)$:
\begin{equation}\label{3.conv}
\lim_{t\to\infty}\|(v(t),\Dt v(t),w(t))-(\bar v,0,\bar w)\|_{[L^p(\Omega)]^3}=0
\end{equation}
for any $1\le p<\infty$.
\end{theorem}
\begin{proof} The proof of that convergence is strongly based on
the perturbation theory of regular attractors and Proposition
\ref{PropA.main} (see Appendix). Indeed, due to Propositions
\ref{Prop3.undis} and \ref{Prop1.wsmooth}, we may assume without
loss of generality that $(v(0),\Dt v(0),w(0)$ belongs to the absorbing ball
$B_R$ in $\Phi_\infty$ with the radius $R$ independent of $\alpha$
and that
\begin{equation}\label{3.sm}
\|w\|_{C_b(\R_+\times\Omega)}+\|\Dt w\|_{C_b(\R_+\times\Omega)}\le C,
\end{equation}
where the constant $C$ is also independent of $\alpha$. Thus, the
first equation of \eqref{1.eqmain}
\begin{equation}\label{3.mainf}
\Dt^2 v(t,x)+\varphi(v(t,x))\Dt v(t,x)+f(v(t,x))=\alpha w(t,x)
\end{equation}
can be treated as an ODE for every fixed $x\in\Omega$.  Moreover,
due to the hyperbolicity assumption \eqref{3.eqhyp} and uniform estimate
\eqref{3.sm} the right-hand side of \eqref{3.mainf} can be treated
as small non-autonomous perturbation of the ODE
\begin{equation}\label{3.mainode}
u''+\varphi(u)u'+f(u)=0.
\end{equation}
Thus, the assumptions of Proposition \ref{PropA.main} hold for
problem \eqref{3.mainf} for every fixed $x\in\Omega$ if
$\alpha\le\alpha_0$ for sufficiently small positive $\alpha_0$.
Due to this Proposition, we have the estimate:
\begin{equation}\label{3.estsep}
\int_0^T|\Dt^2 v(t,x)|+|\Dt v(t,x)|\,dt\le
C_1+C_2\alpha\int_0^T|\Dt w(t,x)|\,dt,
\end{equation}
where the positive constants $C_1$ and $C_2$ are independent of
$T$, $\alpha$ and $x\in\Omega$. Integrating this inequality by
$x\in \Omega$, we arrive at
\begin{equation}\label{3.estl1}
\int_0^T\|\Dt^2 v(t)\|_{L^1}+\|\Dt v(t)\|_{L^1}\,dt\le
C_1|\Omega|+C_2\alpha\int_0^T\|\Dt w(t)\|_{L^1}\,dt.
\end{equation}
In order to estimate the integral into the right-hand side of
\eqref{3.estl1}, we differentiate the second equation of
\eqref{1.eqmain} by $t$, denote $\theta:=\Dt w$, multiply it by
$\sgn \theta(t)$ and integrate over $\Omega$. Then, due to the
Kato inequality, we arrive at
\begin{equation}\label{3.simple}
\Dt\|\theta(t)\|_{L^1}+\|\theta\|_{L^1}\le \|\Dt v(t)\|_{L^1}.
\end{equation}
Integrating this inequality, we have
$$
\|\Dt w(t)\|_{L^1}\le\|\Dt
w(0)\|_{L^1}e^{-t}+\int_0^te^{-(t-s)}\|\Dt v(s)\|_{L^1}\,ds.
$$
Integrating the obtained inequality once more over $t\in[0,T]$ and
using that $\Dt w(0)$ is uniformly bounded, we arrive at
\begin{equation}\label{3.west}
\int_0^T\|\Dt w(t)\|_{L^1}\,dt\le C+\int_0^T\|\Dt v(t)\|_{L^1}\,dt,
\end{equation}
where $C$ is again independent of $\alpha$ and $T$ and the
trajectory.
\par
Inserting \eqref{3.west} into the right-hand side of
\eqref{3.estl1}and assuming that $\alpha$ is small enough, we
finally deduce the following $L^1$-dissipation integral
\begin{equation}\label{3.l1disint}
\int_0^T\|\Dt^2 v(t)\|_{L^1}+\|\Dt v(t)\|_{L^1}+\|\Dt
w(t)\|_{L^1}\,dt \le C,
\end{equation}
where the constant $C$ is independent of $T$.
\par
Thus, we have proved that $(v(t),\Dt v(t),w(t))$ converges to some
$\xi\in\Phi_\infty$ in the $L^1(\Omega)$-norm. Moreover, since we
have the control of the $L^\infty$-norm, the interpolation
inequality gives the convergence in $L^p$ for any $p<\infty$:
\begin{equation}\label{3.conlp}
\lim_{t\to\infty}\|(v(t),\Dt v(t),w(t))-\xi\|_{[L^p(\Omega)]^3}=0.
\end{equation}
Thus, we only need to verify that $\xi\in\Cal R_\alpha$ is an
equilibrium. To this end, we will use the so-called trajectory
approach (see \cite{cv} for the details) and consider positive semi-trajectories   instead of points
 in the phase spaces. Indeed, arguing exactly as in the proof of
 estimate \eqref{1.difest}, but taking the $L^p$-norm instead of
 the $L^\infty$-norm, we see that
 \begin{multline}\label{3.lpdif}
 \|(v_1(t),\Dt v_1(t),w_1(t))-(v_2(t),\Dt
 v_2(t),w_2(t))\|_{[L^p(\Omega)]^3}\le\\\le Ce^{kt}\|(v_1(0),\Dt v_1(0),w_1(0))-(v_2(0),\Dt
 v_2(0),w_2(0))\|_{[L^p(\Omega)]^3}.
\end{multline}
Define now the map $\Bbb S:\Phi_\infty\to
L^\infty(\R_+,\Phi_\infty)$ via the expression
\begin{equation}\label{3.tr}
\Bbb S:(v_0,v_0',w_0)\to (v(\cdot),\Dt v(\cdot),w(\cdot))
\end{equation}
and let $\Cal K_+:=\Bbb S(\Phi_\infty)$. Then, estimate
\eqref{3.lpdif} (together with the obvious fact that $\Cal K_+\subset L^\infty(\R_+,\Phi_\infty)$)
shows that the map $\Bbb S$ realizes a Lipschitz
continuous homeomorphism between spaces $\Phi_\infty$ and $\Cal
K_+$ endowed by the topology of $[L^p(\Omega)]^3$ and
$\Phi_{tr}:=L^\infty(\R_+,[L^p(\Omega)]^3)$ respectively. The solution
semigroup $S(t)$ is conjugated via that homeomorphism to the
semigroup $T(t)$ of temporal shifts on $\Phi_{tr}$:
\begin{equation}\label{3.trdyn}
S(t)=\Bbb S^{-1}\circ T(t)\circ\Bbb S,\ \ \
(T(t)\xi)(s):=\xi(t+s), \ \xi\in\Cal K_+,\ \ t,s\ge0.
\end{equation}
Thus, the convergence \eqref{3.conlp} implies that
$$
T(t)\Bbb S\xi_0:=T(t)\Bbb S(v(0),\Dt v(0),w(0))\to \Bbb S\xi
$$
in the space $\Phi_{tr}$.
\par
 Let us first check that $\Bbb S\xi\in\Cal K_+$, i.e., that the
 limit trajectory $\xi(t)$ solves equation \eqref{1.eqmain}. In other
 words, we need to show that $\Cal K_+$ is closed in $\Phi_{tr}$.
 \par
 To this end,
 we need to show that it is possible to pass to the $\Phi_{tr}$-limit in
 equations \eqref{1.eqmain} in the sense of
 distributions for any sequence $\xi_n(t):=(v_n(t),\Dt v_n(t),w_n(t))$
 converging in $\Phi_{tr}$ to some $\xi(t):=(\bar v(t),\Dt \bar v(t),\bar w(t))$ and bounded
  in $L^\infty(\R_+,\Phi_\infty)$. Indeed,
 the passage to the limit in all linear terms are evident and only
 the passage to the limit in the nonlinear terms $\varphi(v)$ and
 $f(v)$ may a priori be problematic. But it is not the case, since
 convergence in $\Phi_{tr}$ implies the convergence almost
 everywhere (up to extracting a subsequence) and this allows us to
 conclude in a standard way
 that $f(v_n)\rightharpoondown f(\bar v)$ and $\varphi(v_n)\Dt v_n\rightharpoondown\varphi(\bar v)\Dt \bar v$
 (here we have implicitly used that $v_n$ is uniformly bounded in
 $L^\infty$). Thus, the limit function $\xi(t)$ solves indeed
 problem \eqref{1.eqmain}.
 \par
We are now ready to verify that $\xi(t)$ is an equilibrium which
will finish the proof of the theorem. Indeed, due to the
dissipation integral \eqref{1.disint}, we see that
$$
\|\Dt v(s+\cdot)\|_{L^2(\R_+\times\Omega)}+\|\Dt^2
v(s+\cdot)\|_{L^2(\R_+\times\Omega)}+\|\Dt
w(s+\cdot)\|_{L^2(\R_+\times\Omega)}\to0
$$
as $s\to\infty$. Thus, for the limit function $(\bar v(t),\Dt
\bar v(t),\bar w(t))$, we have $\Dt \bar v\equiv\Dt^2 \bar v\equiv\Dt \bar w\equiv0$ and
$\xi$ is indeed an equilibrium. Theorem \ref{Th3.conv} is proved.
\end{proof}

\begin{remark}\label{Rem3.discont} Assume, in addition, that
equation $f(z)=0$ possesses at least two solutions $v_1$ and $v_2$
such that $f'(v_i)>0$, $i=0,1$. Then, $v_i$ will be exponentially
stable equilibria of equation \eqref{3.ODE}.  Let now $v_0(x)$ be
a smooth function such that
$$
v_0(x)=v_1,\ \ x\in\Omega_1,\ \ v_0(x)=v_2,\ \ x\in\Omega_2
$$
for some non-empty $\Omega_i\subset\Omega$ of the nonzero measure.
\par
Finally let us consider the initial data for problem
\eqref{1.eqmain} of the form
$$
\xi_0:=(v_0(x),0,0).
$$
Then, since $\alpha>0$ is small and the equilibria $v_i$ are
exponentially stable, the solution $v(t,x)$ will remain close to
$v_i$ (for $x\in\Omega_i$) for all $t$. This shows that the smooth
trajectory $S(t)\xi_0=(v(t),\Dt v(t),w(t))$  tends as $t\to\infty$
to the {\it discontinuous} equilibrium (in the $L^p$-topology,
according to the last theorem).
\par
This example shows that we cannot extend the assertion of the
theorem to the case $p=\infty$ and obtain the convergence in the
topology of the phase space $\Phi_\infty$. Indeed, if the sequence
of continuous functions converges in $L^\infty$ to some limit
function, this function is automatically continuous. Thus, the
$\omega$-limit set of the above constructed trajectory in the
topology of the phase space is empty:
$$
\omega_{\Phi_\infty}(\xi_0)=\varnothing.
$$
\end{remark}

\begin{remark}\label{Rem3.hyp} It is clear from the proof of
Proposition \ref{Prop3.eqstr} that all of the equilibria $\Cal
R_\alpha$ are hyperbolic in the phase space $\Phi_\infty$. Thus,
we may construct the infinite-dimensional stable and unstable
manifolds for any equilibrium belonging to $\Cal R_\alpha$ if
$\alpha>0$ is small enough. However, it does not help much for the
study of the limit dynamics since, as shown in the previous
proposition, generically, we do not have the stabilization in the topology of
$\Phi_\infty$, but only in a weaker space $[L^p(\Omega)]^3$. And
in this weaker space the solution semigroup $S(t)$ is {\it not
differentiable}. By this reason, we are not able to extract the
{\it exponential} convergence from the hyperbolicity of any
equilibrium and do not know whether or not such exponential
convergence takes place.
\end{remark}

To conclude, we note that, arguing analogously to the proof of Proposition
\ref{Prop3.eqstr}, one can extract  some reasonable
information about the equilibria $\Cal R_\alpha$ even in the case where
$\alpha$ is not small.

\begin{proposition}\label{Prop3.notsmall} Let the assumptions of
Theorem \ref{Th1.main} hold. Assume, in addition, that $(\bar
v,\bar w)\in\R^2$ is a spatially homogeneous hyperbolic (in
$\Phi_\infty$) equilibrium of problem \eqref{1.eqmain}, i.e.,
that the equation
\begin{equation}\label{3.hypeq}
f'(\bar v)\theta-\alpha(-\Dx+1)^{-1}\theta=h
\end{equation}
is uniquely solvable for every $h\in L^\infty(\Omega)$. Assume,
finally, that there exists another constant $\tilde v\ne\bar v$
such that
\begin{equation}\label{3.nm}
f(\bar v)=f(\tilde v)
\end{equation}
(this, of course, may happen only in the case of non-monotone
$f$).
Then, there exists $\delta_0>0$ such that, for any measurable partition
$\Omega=\Omega_1\cup\Omega_2$ on two disjoint sets where
\begin{equation}\label{3.smme}
|\Omega_2|\le\delta_0,
\end{equation}
there exists a hyperbolic equilibrium $(v,w)$ such that $v$ is
close (in the $L^\infty$-metric) to
\begin{equation}\label{3.first}
v_{12}:=\bar v\chi_{\Omega_1}(x)+\tilde v\chi_{\Omega_2}(x).
\end{equation}
\end{proposition}
\begin{proof} We first check that the equation of variations
\begin{equation}\label{3.eqper}
f'(v_{12})\theta-\alpha(-\Dx+1)^{-1}\theta=h
\end{equation}
is uniquely solvable if the measure of $\Omega_2$ is small. To
this end, we construct the approximative solution of this equation
in the form $\tilde\theta:=\theta_0+\hat\theta$ where $\theta_0$
solves equation \eqref{3.hypeq} and
\begin{equation}\label{3.cor2}
\hat\theta(x):=[f'(\bar v)-f'(\tilde v)]\theta_0(x)\chi_{\Omega_2}(x).
\end{equation}
Then, since $\bar v$ is a hyperbolic equilibrium, we have
\begin{equation}\label{3.esth1}
\|\theta_0\|_{L^\infty(\Omega)}\le C\|h\|_{L^\infty(\Omega)}.
\end{equation}
Moreover, the approximate solution $\tilde\theta$ thus
constructed solves
\begin{equation}
f'(v_{12})\tilde\theta-\alpha(-\Dx+1)^{-1}\tilde\theta=\tilde
h:=-\alpha(-\Dx+1)^{-1}\hat \theta.
\end{equation}
Finally, since the measure of $\Omega_2$ is small, we have
$$
\|\hat\theta\|_{L^2}\le
C|\Omega_2|^{1/2}\|\theta_0\|_{L^\infty}\le
C_1\delta_0^{1/2}\|h\|_{L^\infty}
$$
and, consequently,
\begin{equation}\label{3.app}
\|\hat h\|_{L^\infty}\le C\|\hat h\|_{H^2}\le
C'\|\hat\theta\|_{L^2}\le C_2\delta_0^{1/2}\|h\|_{L^\infty},
\end{equation}
where the constant $C_2$ is independent of $h$ and of the
concrete form of the partition $\Omega=\Omega_1\cup\Omega_2$.
\par
Thus, if $\delta_0>0$ is so small that
$C_2\delta_0^{1/2}:=\kappa<1$, the norm of the reminder
$\|\tilde h\|_{L^\infty}$ is estimated $\kappa\|h\|_{L^\infty}$
with $\kappa<1$. Then, the standard iteration process gives the
desired solution $\theta$ of equation \eqref{3.eqper} together
with estimate
\begin{equation}\label{3.esth2}
\|\theta\|_{L^\infty}\le C\|h\|_{L^\infty}
\end{equation}
with the constant $C$ independent of $\delta_0\to0$. The
uniqueness of a solution can be obtained in a standard way using
the observation that the operator $f'(v_{12})-\alpha(-\Dx+1)^{-1}$
is self-adjoint in $L^2$.
\par
It is now not difficult to finish the proof of the proposition.
Indeed, we seek for the desired equilibrium $(v,w)$ in the form
$$
v(x)=v_{12}(x)+\theta(x),
$$
where $\theta$ is a small corrector which should satisfy the
equation
\begin{equation}\label{3.eqequ}
f(v_{12}+\theta)-f(v_{12})-\alpha(-\Dx+1)^{-1}\theta=\hat h:=\alpha(-\Dx+1)[(\bar
v-\tilde v)\chi_{\Omega_2}]
\end{equation}
and, arguing as before, we see that
$$
\|\hat h\|_{L^\infty}\le C|\Omega_2|^{1/2}.
$$
Applying now the implicit function theorem to equation
\eqref{3.eqequ}, we establish the existence of a unique solution
$\theta$,
$$
\|\theta\|_{L^\infty}\le C_1|\Omega_2|^{1/2}
$$
measure of $\Omega_2$ is small enough. Proposition
\ref{Prop3.notsmall} is proved.
\end{proof}

\begin{remark}\label{Rem3.last} Although we are not able neither to give a
complete description of the equilibria set $\Cal R_\alpha$  nor to verify the stabilization if
$\alpha$ is not small, we see that, under the assumptions of the
last proposition (which are, in a sense, natural for the
non-monotone case), the set of equilibria is not compact in
$\Phi_\infty$ and is not closed in the weak-star topology of the
phase space. These facts do not allow to extend the global attractor
theory for the non-monotone case.
\par
Note also that, although we formulate (for simplicity) Proposition
\ref{Prop3.notsmall} for the case of {\it spatially-homogeneous}
hyperbolic equilibrium $(\bar v,\bar w)$ it can be easily extended
to the case of non-ho\-mo\-ge\-ne\-ous equilibria. This shows that the
conclusion of Proposition \ref{Prop3.notsmall} is somehow "generic"
for the non-monotone case.
\end{remark}

\section{Appendix. Estimation of  time derivatives for non-autonomous perturbations
of regular attractors.}\label{s4}
The aim of that appendix is to verify the auxiliary estimate for
non-autonomous perturbations of regular attractors. To be more
precise, consider an ODE in $\R^n$:
\begin{equation}\label{A.ODE}
u'(t)=F(u(t)), \ u(0)=u_0
\end{equation}
for some $F\in C^2(\R^n,\R^n)$. We assume that, for every
$u_0\in\R^n$, this equation is globally (for $t\ge0$) solvable
and the associated semigroup $S(t)u_0:=u(t)$ is dissipative, i.e.,
\begin{equation}\label{A.dis}
\|S(t)u_0\|\le Q(\|u_0\|)e^{-\beta t}+C_*
\end{equation}
for some positive $\beta$ and $C_*$ and monotone $Q$. Therefore,
equation \eqref{A.ODE} possesses a global attractor $\Cal A$ in
$\R^n$. Our main assumption is that this attractor is {\it
regular} in the sense of Theorem \ref{Th2.areg}, i.e., all of the
equilibria $u_0\in\Cal R$ are hyperbolic, every trajectory,
belonging to the attractor $\Cal A$ is a heteroclinic orbit
connecting two {\it different} equilibria and the attractor $\Cal
A$ is a finite union of the unstable manifolds $\Cal M^+_{u_0}$ associated with
the equilibria $u_0\in\Cal R$:
\begin{equation}\label{A.rstr}
\Cal A=\cup_{u_0\in\Cal R}\Cal M^+_{u_0}.
\end{equation}
 Finally, we assume that the so-called no-cycle condition is satisfied, i.e., the attractor $\Cal A$ does not contain any heteroclinic cycles. As known, that is always true in the case when \eqref{A.ODE} possesses a global Lyapunov function.
\par
Consider now the following small non-autonomous perturbation of
equation \eqref{A.ODE}
\begin{equation}\label{A.NODE}
u'=F(u)+h(t),\ \ u(0)=u_0,\ \ t\ge0,
\end{equation}
where the non-autonomous external force is uniformly small:
\begin{equation}\label{A.small}
\|h\|_{W^{1,\infty}(\R_+,\R^n)}\le \eb\ll1.
\end{equation}
The main result of this appendix is the following estimate.

\begin{proposition}\label{PropA.main} Let the above assumptions
hold and let the external force $h\in W^{1,\infty}(\R_+)$ satisfy
estimate \eqref{A.small} for sufficiently small $\eb>0$.
 Then, any solution $u(t)$ of the perturbed problem \eqref{A.NODE}
 satisfies the following estimate:
 \begin{equation}\label{A.L1}
 \int_0^T\|u'(t)\|\,dt\le C_1+C_2\int_0^T\|h'(t)\|\,dt,
 \end{equation}
 where the positive constants $C_1$ and $C_2$ depend only on the
 norm of $u(0)$ and are independent of $T$ and the concrete choice
 of $u(0)$ and $h(t)$.
 \end{proposition}
 \begin{proof} Indeed, using the standard regular attractor
 perturbation arguments, one can check that for every bounded set
 $B$ of $\R^n$ and every $\delta>0$, there exist $T=T(B,\delta)$
 and $\eb_0=\eb_0(B,\delta)$ such that, for every $\eb\le\eb_0$ and every trajectory
 $u(t)$ starting from $B$, we can find a sequence $u_,\cdots,u_N$
 of {\it different} equilibria $u_i\in\Cal R$ (of problem \eqref{A.ODE}!) and a sequence of
 times
 $$
 0=T_0^+\le T_1^-\le T_1^+<T_2^-\le T_2^+<\cdots<T_{N}^-<T_N^+=\infty
 $$
 such that
 \begin{equation}\label{A.regseq}
 u(t)\in \Cal O_\delta(u_i),\ \ \ t\in(T_i^-,T_i^+),\ \ \
 T_{i}^--T_{i-1}^+\le T,\ i=1,\cdots, N.
 \end{equation}
In other words, the sequence of equilibria $u_i$ and the values of $T_l^{\pm}$ depend on the concrete
choice of $u(0)$ and $h$, but the number $N$ of equilibria is
bounded by the whole number $\#\Cal R<\infty$ of possible
equilibria (since the equilibria must be different) and the
differences $T_{i}^+-T_{i-1}^-$ are also uniformly bounded by $T$,
see \cite{bv,efz,vz} for the details.
\par
Thus, any trajectory starting from $B$ spends at most time $T_{out}:=\#\Cal
R\cdot T$ outside of the $\delta$-neighbourhood $\Cal
O_\delta(\R)$ of the equilibria set $\Cal R$ and this time depends
only on $B$ and $\delta$. By this reason, the part of the
trajectory, lying outside of $\Cal O_\delta(\R)$ gives only a {\it
finite and uniformly bounded} impact to the integral \eqref{A.L1}
(which can be included into the constant $C_1$). So, we only need
to estimate the left-hand side of \eqref{A.L1} for the case where
$u(t)$ belongs to a small neighbourhood of a single fixed
equilibrium $u_0\in\Cal R$ only.
\par
To this end, we will use the hyperbolicity assumption on $u_0$.
Indeed,
 the implicit function theorem implies the existence of $\eb_0>0$ and $\delta>0$ such that, for every
 $\eb\le\eb_0$, there exists a unique solution $U_{u_0,h}(t)$
of \eqref{A.NODE}
 belonging to the $\delta$-neighbourhood of $u_0$ for all $t$.
 Moreover, this solution, in a fact, belongs to the
 $C\eb$-neighbourhood of $u_0$ and the following estimate holds:
 \begin{equation}\label{A.prime}
 |U'_{u_0,h}(t)|\le C\int_{\R}e^{-\kappa|t-s|}| h'(s)|\,ds,
 \end{equation}
 where the constant $C$ and the hyperbolicity exponent $\kappa$
 are independent of the concrete choice of $u_0\in\Cal R$ and the
 external force $h$ satisfying \eqref{A.small}, see \cite{gv,vz}.
\par
Furthermore, since $u_0$ is hyperbolic, the trajectory
$U_{u_0,h}(t)$ will be also hyperbolic and we will have an
exponential dichotomy in a small $\delta$-neighbourhood of
$U_{u_0,h}$. In particular, every trajectory $u(t)$ belonging to
$\Cal O_\delta(u_0)$ for $t\in[0,S]$, $S\gg1$, will tend
exponentially to $U_{u_0,h}(t)$ inside of the interval
\begin{equation}\label{A.exsmall}
|u(t)-U_{u_0,h}(t)|+|u'(t)-U'_{u_0,h}(t)|\le
C(e^{-\kappa|t|}+e^{-\kappa|S-t|})
\end{equation}
and $C$ and $\kappa$ are independent of the concrete choice of
$u$ and $h$,
see \cite{gv,vz} for the details. Therefore,
\begin{equation}\label{A.est1}
\int_0^t|u'(s)|\,ds\le C+\int_0^t|U'_{u_0,h}(s)|\,ds
\end{equation}
for $t\in[0,S]$ and $u(t)\in\Cal O_\delta(u_0)$.
\par
Thus, we have proved that
\begin{equation}\label{A.est2}
\int_0^T|u'(s)|\,ds\le C_1+C_2\sum_{u_0\in\Cal
R}\int_0^T|U'_{u_0,h}(s)|\,ds,
\end{equation}
where the constants $C_i$ depend only on the radius of $B$.
\par
In order to deduce \eqref{A.L1} from \eqref{A.est2}, we will use
estimate \eqref{A.prime}. Indeed, integrating it over $t\in[0,T]$
and using that $h$ can be extended for $t\le0$ by zero,
we have
\begin{equation}\label{A.est3}
\int_0^T|U'_{u_0,h}(t)|\,dt\le
C\int_0^T|h'(t)|\,dt+C\int_{T}^\infty e^{-\kappa|t-T|}|h'(s)|\,ds\le
C_1+C\int_0^T|h'(t)|\,dt.
\end{equation}
Inserting this estimate into the right-hand side of
\eqref{A.est2}, we obtain \eqref{A.L1} and finish the proof of the
proposition.
\end{proof}

\begin{remark} As we can see from the proof, estimate \eqref{A.L1}
has a general nature which is not restricted by the class of
ordinary differential equations. However, since we use it in the
paper only for the ODE case, in order to avoid the technicalities
related with the formulation of a "general" PDE, we restrict our
consideration to the case of a ODE as well.
\end{remark}

\def\bibname{\textbf{\large References}}


\begin{thebibliography}{Ba87}
\bibitem[1]{AB}
{\au G.~Andrews and J.~M.~Ball}, {\ti Asymptotic behaviour and changes of phase in one-dimensional nonlinear
viscoelasticity}, {\jou Journal of Differential Equations},  \no{44}{1982}{306--341}.
\bibitem[2]{an}
        \textsc{S. Agmon, L. Nirenberg}:
        Lower bounds and uniqueness theorems for solutions of
        differential equations in a Hilbert space,
        \emph{Comm. Pure Appl. Math.}, \textbf{20}, 1967, 207--229.

\bibitem[3]{bv}
        \textsc{A.V. Babin, M.I. Vishik}:
        \emph{Attraktory evolyutsionnykh uravnenii} (English transl.
        Stud. Math. Appl., 25, North Holland, Amsterdam, 1992),
        ``Nauka'', 1989.

\bibitem[4]{pata}
{\au V. Belleri and V. Pata}, {\ti Attractors for semilinear strongly damped wave equations on $\Bbb R\sp 3$},
{\jou  Discrete Contin. Dynam. Systems  7} \no{4}{719--735}{2001}.

\bibitem[5]{cv}
{\au V.V. Chepyzhov and M.I. Vishik},
{\bk Attractors of equations of mathematical physics},
\eds{Amer. Math. Soc.}{Providence, RI}{2002}
\bibitem[6]{Y1}
{\au L.H. Chuan, T. Tsujikawa and A. Yagi},
{\ti Asymptotic behavior of solutions for forest kinematic model},
{\jou  Funkcial. Ekvac.  49}, \no{3}{427--449}{2006}.
\bibitem[7]{Y2}
{\au  L. H. Chuan and A. Yagi},
{\ti Dynamical system for forest kinematic model},
{\jou Adv. Math. Sci. Appl.},
\no{16}{343--409}{2006}

\bibitem[8]{efz}
{\au M. Efendiev and  S. Zelik},
{\ti The regular attractor for the
reaction-diffusion system with a nonlinearity rapidly oscillating in time and its averaging}
{\jou  Adv. Differential Equations  8}, \no{6}{673--732}{2003}  (2003)



\bibitem[9]{p}
{\au S. Gatti and V. Pata}
{\ti A one-dimensional wave equation with nonlinear damping},
{\jou  Glasg. Math. J.  48}, \no{3}{419--430}{2006}.

\bibitem[10]{gv}
{\au A. Goritsky and M. Vishik},
{\ti Integral manifolds for nonautonomous equations},
{\jou  Rend. Accad. Naz. Sci. XL Mem. Mat. Appl. (5)}, \no{21}{ 109--146}{1997}

\bibitem[11]{hale}
{\au J.K. Hale},
{\bk Asymptotic behavior of dissipative systems},
\eds{Math. Surveys Monographs, Vol. 25, Amer. Math. Soc.}{Providence,
  RI}{1988}
\bibitem[12]{henry}
\textsc{ D. Henry}:
Geometric theory of semilinear parabolic equations.
\emph{Lecture Notes in Mathematics}, 840. Springer-Verlag, Berlin-New York, 1981.
\bibitem[13]{KZ}
{\au P. Krej\v ci' and S. Zheng},
{\ti Pointwise asymptotic convergence of solutions for a phase separation model},
{\jou  Discrete Contin. Dyn. Syst.  16}, \no{1}{1--18}{2006}.
\bibitem[14]{K}
{\au  Yu. Kuznetsov, M. Antonovsky, V. Biktashev and A. Aponina},
{\ti A cross-diffusion model of forest boundary dynamics},
{\jou J. Math. Biol. 32}, \no{}{219--232}{1994}

\bibitem[15]{LSU}
         \textsc{O.A. Ladyzhenskaya, V.A. Solonnikov, N.N. Uraltseva}:
         \emph{Linear and Quasilinear Equations of Parabolic Type}, Nauka,
         1967.
\bibitem[16]{Y3}
{\au H. Nakata},
{\ti Numerical simulations for forest boundary dynamics model},
Master's Thesis, Osaka University, 2004.

\bibitem[17]{pz1}
{\au V. Pata and S.  Zelik}, {\ti Attractors and their regularity for 2-D
 wave equations with nonlinear damping}, {\jou  Adv. Math. Sci. Appl.  17},
\no{1}{225--237}{2007}.

\bibitem[18]{pz2}
{\au V. Pata and S.  Zelik}, {\ti Global and exponential attractors for 3-D
 wave equations with displacement dependent damping}, {\jou  Math. Methods Appl. Sci.  29},
\no{11}{1291--1306}{2006}.
 \bibitem[19]{pego}
 {\au R. Pego},  {\ti Phase transitions in one-dimensional nonlinear viscoelasticity: admissibility and stability}, {\jou Arch. Rational Mech. Anal.}, \no{97}{353--394}{1987}.
\bibitem[20]{vz}
{\au M. Vishik and S. Zelik},
{\ti  Regular attractors and their non-autonomous perturbations},
preprint, University of Surrey,
http://personal.maths.surrey.ac.uk/st/S.Zelik/publications/publ.html/ellv2.pdf


\end{thebibliography}
\end{document}